\documentclass[a4paper,11pt,reqno]{amsart}
\usepackage{mathtools,amsmath,amssymb,amsfonts,amsthm,mathrsfs}
\usepackage[pdftex]{graphicx}
\usepackage{graphicx}

\DeclareGraphicsExtensions{.bmp,.png,.pdf,.jpg,}

\usepackage{bm}
\usepackage{verbatim}
\usepackage{color}
\usepackage{mathrsfs}
\usepackage{pgfplots}

\usepackage{lmodern}
\usepackage[utf8]{inputenc}
\usepackage[T1]{fontenc}
\usepackage[english]{babel}
\usepackage{bbm}
\usepackage{pgfplots}
\graphicspath{{Figures/}}
\usepackage{caption}
\usepackage{subcaption}
\usepackage{microtype}

\usepackage{booktabs}
\usepackage{multirow}
\usepackage{rotating,tabularx}
\usepackage{adjustbox}
\usepackage{enumitem}
\usepackage[hidelinks]{hyperref}
\usepackage[normalem]{ulem}

\def\R{{\mathbb R}}
\def\N{{\mathbb N}}

\def\a{\alpha}
\def\b{\beta}
\def\t{\theta}

\def\r{\varrho}
\def\g{\gamma}
\def\s{\sigma}
\def\t{\theta}
\def\l{\lambda}
\def\p{\partial}
\def\O{\Omega}
\def\e{\varepsilon}
\def\v{\varphi}

\def\o{\omega}
\def\mc{\mathcal}
\def\mf{\mathfrak}
\def\ua{\uparrow}
\def\da{\downarrow}

\numberwithin{equation}{section}

\theoremstyle{definition}

\theoremstyle{plain}
\newtheorem{theorem}{Theorem}[section]

\newtheorem{lemma}{Lemma}[section]
\newtheorem{corollary}{Corollary}[section]

\newtheorem{conjecture}{Conjecture}[section]

\theoremstyle{definition}

\begin{document}

\title[Global bifurcation diagrams of positive solutions]
{Global bifurcation diagrams of positive solutions for a class of 1-D superlinear indefinite problems}

\author{M. Fencl}
\address{Department of Mathematics and NTIS, Faculty of Applied Sciences,
University of West Bohemia, Univerzitn\'i 8, 30100, Plzen, Czech Republic}
\email{fenclm37@ntis.zcu.cz}
\author{J. L\'{o}pez-G\'{o}mez}
\address{Institute of Inter-disciplinar Mathematics (IMI) and Department of Analysis and Applied Mathematics, Complutense University of Madrid, Madrid 28040, Spain}
\email{Lopez$\_$Gomez@mat.ucm.es}

\maketitle

\begin{abstract}
This paper analyzes the structure of the set of positive solutions of a class of one-dimensional superlinear indefinite bvp's. It is  a paradigm of how mathematical analysis aids the numerical study of a problem, whereas simultaneously its numerical study confirms and illuminates the analysis. On the analytical side, we establish the fast decay of the positive solutions as $\l\da -\infty$  in the region where $a(x)<0$ (see \eqref{1.1}), as well as  the decay of the solutions of the parabolic counterpart of the model (see \eqref{1.2}) as $\l\da -\infty$ on any subinterval of $[0,1]$ where $u_0=0$, provided $u_0$ is a subsolution of \eqref{1.1}.
This result provides us with a proof of a conjecture of \cite{GRLGJDE} under an additional
condition of a dynamical nature.  On the numerical side, this paper ascertains the global structure of the set of positive solutions on some paradigmatic prototypes whose intricate behavior is far from predictable from  existing analytical results.
\vspace{0.1cm}

\vspace{0.1cm}

\vspace{0.1cm}

\noindent Partially supported by the Research Grant  PGC2018-097104-B-IOO of the Spanish Ministry of Science, Innovation and Universities,  and the Institute of Inter-disciplinar Mathematics (IMI) of Complutense University. M. Fencl has been supported by the project SGS-2019-010 of the University of West Bohemia, the project 18-03253S of the Grant Agency of the Czech Republic and the project LO1506 of the Czech Ministry of Education, Youth and Sport.
\end{abstract}

\section{Introduction}

\noindent In this paper we study, both analytically and numerically,  the global structure of the  bifurcation diagram of positive solutions of the semilinear boundary value problem
\begin{equation}
\label{1.1}
		\left\{\begin{aligned}
			-u'' &= \lambda u + a(x)u^{2}\quad\text{in }(0,1),\\
			u(0) &= u(1) = 0,
		\end{aligned}\right.
\end{equation}	
where $a\in C[0,1]$ is a real function that changes the sign in $(0,1)$ and $\lambda\in\mathbb{R}$ is regarded as a bifurcation parameter. Moreover, we analyze the dynamics of its parabolic counterpart
\begin{equation}
\label{1.2}
		\left\{\begin{array}{ll}
			\frac{\p u}{\p t}-\frac{\p^2 u}{\p x^2}  = \lambda u + a(x)u^{2}, &
        \quad t>0, \;\; x\in(0,1),\\ u(0,t) = u(1,t) = 0, & \quad t>0,\\
        u(x,0)=u_0(x), & \quad x\in [0,1],
		\end{array}\right.
\end{equation}	
for some significative choices of the initial data $u_0\gneq 0$, i.e., $u_0\geq 0$, $u_0\neq 0$.
\par
In our numerical experiments we have used the special choices
\begin{equation}
\label{1.3}
  a(x) := \sin[(2n+1) \pi x],\qquad n\in\{1,2,3\},
\end{equation}
and
\begin{equation}
\label{1.4}
	a(x):= \left\{ 	\begin{array}{ll} {\mu} \sin(5\pi x) & \quad \text{ if }\;\; x\in [0,0.2)\cup(0.8,1],
\\[1ex]  \sin(5\pi x) & \quad\text{ if }\;\; x\in [0.2,0.8],  \end{array}\right.
\end{equation}
where $\mu\geq 1$ is regarded as a secondary bifurcation parameter. In these examples, the graph of
 $a(x)$ has  $n+1$ positive and $n$ negative bumps.
\par
Since $a(x)$ changes the sign, \eqref{1.1} is a superlinear indefinite problem. These problems have attracted a lot of attention during the last three decades. Some significant monographs dealing with them are those of  Berestycki, Capuzzo-Dolcetta and Nirenberg \cite{BCDNa,BCDNb}, Alama and Tarantello \cite{AT}, Amann and L\'{o}pez-G\'{o}mez \cite{ALG}, G\'{o}mez-Re\~{n}asco and L\'{o}pez-G\'{o}mez \cite{GRLGJDE,GRLGDIE}, Mawhin, Papini and Zanolin \cite{MPZ}, L\'{o}pez-G\'{o}mez, Tellini and Zanolin \cite{LGTZ},   L\'{o}pez-G\'{o}mez and Tellini \cite{LGT}, Feltrin and Zanolin, as well as Chapter 9 of L\'{o}pez-G\'{o}mez \cite{LG16}, the recent monograph  of Feltrin \cite{Fe},  and the list of references there in. Superlinear indefinite  problems have been recently introduced in the context of quasilinear elliptic equations by L\'{o}pez-G\'{o}mez,  Omari and Rivetti \cite{LGORa,LGORb}, and L\'{o}pez-G\'{o}mez and Omari \cite{LGOa,LGOb,LGOc}.
\par
Thanks to  Amann and L\'{o}pez-G\'{o}mez \cite{ALG}, it is already known that \eqref{1.1} possesses a component of positive solutions, $\mathscr{C}^+\subset \R \times \mc{C}[0,1]$, such that $(\pi^2,0)\in\mathscr{\bar C}^+$, i.e., $\mathscr{C}^+$ bifurcates from $u=0$ at $\l=\pi^2$. Moreover, $\mathscr{C}^+$ is unbounded in $ \R \times \mc{C}[0,1]$, and  \eqref{1.1} cannot admit a positive solution for sufficiently large $\l>\pi^2$. Furthermore, by the existence of universal a priori bounds uniform on compact subintervals of $\l\in\R$ for the positive solutions
of \eqref{1.1}, $(-\infty,\pi^2)\subset \mc{P}_\l(\mathscr{C}^+)$, where $\mc{P}_\l$ stands for the $\l$-projection operator defined by
$$
  \mc{P}_\l (\l,u)=\l,\qquad (\l,u)\in \R \times \mc{C}[0,1].
$$
Actually, according to G\'{o}mez-Re\~{n}asco and L\'{o}pez-G\'{o}mez \cite{GRLGJDE,GRLGDIE}, either $\mc{P}_\l(\mathscr{C}^+)=(-\infty,\pi^2)$, or there exists $\l_t>\pi^2$ such that  $\mc{P}_\l(\mathscr{C}^+)=(-\infty,\l_t]$. Moreover, \eqref{1.1} admits  some stable
positive solution if, and only if, $\l\in (\pi^2,\l_t]$, and, in such case,  the stable solution
is unique, and it equals the minimal positive solution of \eqref{1.1}. The fact that  $\l_t$ is turning point is emphasized by its subindex.
\par
Besides the (optimal) multiplicity result of Amann and  L\'{o}pez-G\'{o}mez \cite{ALG},
establishing that \eqref{1.1} has, at least, two positive solutions for every $\l\in (\pi^2,\l_t)$ if $\l_t >\pi^2$, there are some others multiplicity results by Gaudenzi, Habets and Zanolin \cite{GHZ} later generalized by Feltrin and Zanolin \cite{FZ} and Feltrin \cite{Fe}. Precisely, according to Corollary 1.4.2 of Feltrin \cite{Fe}, which extends \cite[Th.2.1]{GHZ}, setting $a=a^+ - \mu a^-$, there exists $\mu_c>0$ such that, for every $\mu > \mu_c$, the problem
\begin{equation}
\label{1.5}
		\left\{\begin{aligned}
			-u'' &= (a^+-\mu a^-)u^2\quad\text{in }(0,1),\\
			u(0) &= u(1) = 0,
		\end{aligned}\right.
\end{equation}	
possesses, for every $\mu >\mu_c$, at least, $2^{n+1}-1$ positive solutions if $a(x)$ is given by \eqref{1.3}. However, this result does not solve the conjecture of  G\'{o}mez-Re\~{n}asco and L\'{o}pez-G\'{o}mez \cite{GRLGJDE} according with it there is a  $\l_c<\pi^2$ such that, for every
$\l < \l_c$, \eqref{1.1} possesses, at least,
\begin{equation*}
  \sum_{j=1}^{n+1} \Big(\!\! \begin{array}{c}  n+1 \\ j \end{array}\!\!\Big) =2^{n+1}-1
\end{equation*}
positive solutions; among them, $n+1$ with a single peak around each of the maxima of $a(x)$, $\frac{(n+1)n}{2}$  with two peaks, and, in general,  $\frac{(n+1)!}{j!(n+1-j)!}$ with $j$ peaks for every $j\in\{1,...,n+1\}$. For instance, when $n=1$, then $a(x) = \sin(3\pi x)$ and, according to our numerical experiments,  \eqref{1.1} indeed has, for sufficiently negative $\l$,  three positive solutions: one with a bump on the left, one with a bump on the right, and another one with two bumps (see Figure \ref{FIG:sin3 sol examples}). Note that, essentially, Corollary 1.4.2 of Feltrin \cite{Fe} establishes that \eqref{1.1} possesses $2^{n+1}-1$ positive solutions provided $\l=0$ and $\|a^-\|_\infty$ is sufficiently large, though it does not give any information for $\l<0$. Thus, after two decades, the conjecture of \cite{GRLGJDE} seems to remain open. For the purpose of this paper, we formulate here this conjecture in the following manner:
\begin{conjecture}
	\label{con3.1}
	Suppose that $a(x)$ possesses $n+1$ intervals where it is positive separated away by $n$ intervals where it is negative. Then, there exists $\l_c<\pi^2$ such that, for every $\l <\l_c$, the problem \eqref{1.1} admits, at least, $2^{n+1}-1$ positive solutions.
\end{conjecture}
The main goal of this paper is to gain some insight, on this occasion of a dynamical nature, into that conjecture and to face, by the first time, the ambitious problem of ascertaining the global topological structure of the set of positive solutions of \eqref{1.1}. As a direct consequence of our numerical experiments for the special choice \eqref{1.4}, it becomes apparent the optimality of  \cite[Cor. 1.4.2]{Fe}, in the sense that, for sufficiently small $\mu>0$, the problem  \eqref{1.5} might have less than $2^{n+1}-1$ positive solutions.
\par
Throughout most of this paper, we will assume that, much like for the special choice \eqref{1.3}, $a(x)$ satisfies
\vspace{0.2cm}
\begin{enumerate}
\item[{($\mathrm{H}_a$)}] The open sets
$$
  \O_-:= a^{-1}((-\infty,0))\quad \hbox{and}\quad \O_+:= a^{-1}((0,\infty))
$$
consist of finitely many (non-trivial) intervals, $I^-_j$, $j\in\{1,...,r\}$, and $I^+_i$, $i\in \{1,...,s\}$,
respectively,  and $a$ vanishes at the ends of these intervals in such a way that each interior interval $I_i^\pm$ is surrounded by two intervals of the form
$I^\mp_j$, much like it happens with the special choice \eqref{1.3}. In such case, we will denote,
$I^-_j=(\a_j,\b_j)$, with $\a_j<\b_j$ for all $j\in\{1,...,r\}$, and $I_i^+=(\g_i,\r_i)$,
with $\g_i<\r_i$ for all $i\in\{1,...,s\}$.
\end{enumerate}
\vspace{0.2cm}

As these intervals are adjacent and interlacing, $|r-s|\leq 1$. Under this assumption, our main analytical results can be summarized as follows. Theorem \ref{th3.1} establishes that, for any family of positive solutions of \eqref{1.1}, $\{(\l,u_\l)\}_{\l<0}$,
\begin{equation}
\label{1.6}
  \lim_{\l\da -\infty}u_\l (x)=0 \quad \hbox{for all}\;\; x\in \O_-=\bigcup_{j=1}^r I^-_j
\end{equation}
uniformly in compact subsets of $\O_-$. Theorem \ref{th3.2} establishes that there exists $T>0$ such that, as soon as $u_0\geq 0$ is a subsolution of \eqref{1.1}, the unique solution of \eqref{1.2}, denoted by $u(x,t;u_0,\l)$, is defined in $[0,T]$ as $\l\da -\infty$ and satisfies, for every $t\in [0,T]$ and
$x\in \O_-$,
\begin{equation}
\label{1.7}
  \lim_{\l\da -\infty} u(x,t;u_0,\l)=0.
\end{equation}
Moreover, this behavior is inherited by the intervals $I_i^+$ where $u_0=0$, as soon as $u_0$ also vanishes at the adjacent $I^-_j$'s, as established by the next result.

\begin{theorem}
\label{th1.1}
Suppose that $u_0\gneq 0$ is a subsolution of \eqref{1.1} such that $u_0=0$ on some interior $I_i^+$ for some $i\in\{1,...,s\}$, as well as on its adjacent intervals, say $I_j^-$ and $I_{j+1}^-$, i.e.,
$$
  u_0=0\quad \hbox{in}\;\; I_j^-\cup I_i^+\cup I_{j+1}^-=(\a_j,\b_j)\cup (\b_j,\a_{j+1})\cup (\a_{j+1},\b_{j+1}).
$$
Then, there exists $T=T(u_0)>0$ such that,  for every $t\in [0,T]$ and $ x\in (\a_j,\b_{j+1})$,
\begin{equation}
\label{1.8}
  \lim_{\l\da-\infty}u(x,t;u_0,\l)=0.
\end{equation}
Moreover, \eqref{1.8} holds uniformly in compact subintervals of $(\a_j,\b_{j+1})$. A similar result holds true for $I^+_1=(0,\r_1)$ and $I^+_{s}=(\g_s,1)$.
\end{theorem}

By a simple combinatorial argument, Theorem \ref{th1.1} provides us with a further evidence supporting the  conjecture of \cite{GRLGJDE}. Actually, it proves it under an additional assumption. Indeed, suppose that $a(x)$ satisfies (H$_a$) with $r=n\geq 1$ and $s=n+1$ and, for every $i\in \{1,...,n+1\}$,  let $\t_{\{\l,i\}}$ be a positive solution of
\begin{equation}
\label{1.9}
		\left\{\begin{aligned}
			-u'' &= \lambda u + a^+(x)u^{2}\quad\text{in }(\g_i,\r_i),\\
			u(\g_i) &= u(\r_i) = 0,
		\end{aligned}\right.
\end{equation}	
and consider the subsolution of \eqref{1.1} defined through
$$
  u_0:= \left\{ \begin{array}{lll} \t_{\{\l,i\}} &\quad \hbox{in}\;\; I_i^+, & \qquad
  i\in\{1,...,n+1\}, \\[1ex] 0 &\quad \hbox{in}\;\; I_j^-, & \qquad
  j\in\{1,...,n\}. \end{array}\right.
$$
Suppose that $u(x,t;u_0,\l)$ is globally bounded in time as $\l\da -\infty$. Then, Theorem \ref{th3.3} shows that there exists $\l_c<0$ such that \eqref{1.1} has $2^{n+1}-1$ positive solutions for every  $\l<\l_c$. But the extremely challenging problem of ascertaining whether or not
$u(x,t;u_0,\l)$ is globally bounded as $\l\da -\infty$ remains open
in this paper.
\par
Throughout this paper,
 for any $a(x)$ with $n+1$ positive bumps separated away by
negative ones, we use a code with $n+1$ digits in  $\{0,1\}$, where $1$ means that the solution has a bump localized at the nodal interval indicated by its position in the code, whereas $0$ means that no bump in that position exists. Thus, when, e.g., $a(x)= \sin(3\pi x)$,  we have positive solutions in Figure \ref{FIG:sin3 sol examples} represented by $2$-digit codes, where $00$ stands for the trivial solution, $10$ stands for a solution with a single bump on the left, $01$ stands for a solution with one single bump on the right, and $11$ stands for a positive solution with both bumps around each of the interior maxima of $a(x)$. At the end of this code, called the \emph{type of the solution} in this paper, we will always add a positive integer within  parenthesis, the \emph{Morse index}, i.e., the dimension of the unstable manifold of the positive solution as a steady state of the associated parabolic problem \eqref{1.2}.  The dimension of the unstable manifold of a given steady state solution, say $u$, equals  the number of negative eigenvalues, $\tau$, of the linearized problem
\begin{equation}
\label{1.10}
		\left\{\begin{aligned}
		-v'' & =\l v + 2 a(x)u(x) v + \tau v  \quad\text{ in }(0,1),\\
		v(0) & = v(1) = 0.
		\end{aligned}\right.
\end{equation}	
\par
Although there is a huge amount of literature on bump and multi-bump solutions for nonlinear 
Schr\"{o}dinger equations (see, e.g., Ambrosetti, Badiale and Cingolani \cite{ABC}, del Pino and Felmer \cite{DPFa,DPFb},  Dancer and Wei \cite{DW}, Wei \cite{Wei}, Wang and Zeng \cite{WZ}, Byeon and Tanaka \cite{BT}, among many others), in the existing literature $a(x)$ always is a positive function. So, none of these results can be applied in our general context, which explains why the conjecture of \cite{GRLGJDE} remains open.
\par 
This paper is organized as follows. Section 2 collects the available information
concerning the global structure of the set of positive solutions of \eqref{1.1} paying attention to the detail as some of these results, rather topological,  are not well known by experts yet.
In Section 3 we prove
Theorem \ref{th1.1} by using the theory of metasolutions (see, e.g., \cite{LG16}). Then,
we infer from it Theorem \ref{th3.3}. Finally,  in Sections 4, 5 and 6  we present and discuss the results of our numerical experiments in cases $n=1$, $n=2$ and $n=3$, respectively. In Section 7 we analyze the more sophisticate case when $a$ is given by \eqref{1.4} using $\mu\geq 1$ as the secondary bifurcation parameter. In Section 8 we shortly discuss the necessary numerics to implement the numerical experiments of this paper.  The paper ends with a final discussion carried out in Section 9. Discussing the results of our numerical experiments in this short general presentation seems inappropriate. The readers should enjoy them in their own sections.

\section{Global structure of the component $\mathscr{C}^+$ }

\noindent This section analyzes the local and global behaviors of the component of positive solutions $\mathscr{C}^+$ introduced in Section 1. The next result, of a technical nature, allows us to express, equivalently, \eqref{1.1} as a fixed point equation for a compact operator. As the proof is elementary, we omit it herein.

\begin{lemma}
\label{le2.1}
For every $f \in \mc{C}[0,1]$, the function
\begin{equation}
\label{2.1}
  u(x) =\int_0^x (s-x)f(s)\,ds - x \int_0^1 (s-1)f(s)\,ds
\end{equation}
provides us with the unique solution of the linear boundary value problem
\begin{equation}
\label{2.2}
\left\{ \begin{array}{l} -u''= f \quad \hbox{in}\;\; [0,1], \\
  u(0)=u(1)=0. \end{array}\right.
\end{equation}
\end{lemma}

According to Lemma \ref{le2.1}, we introduce the linear integral operator $K: \mc{C}[0,1]\to \mc{C}^2[0,1]$ defined, for every $f \in \mc{C}[0,1]$, by
\begin{equation}
\label{2.3}
K f(x):= \int_0^x (s-x)f(s)\,ds - x \int_0^1 (s-1)f(s)\,ds, \quad x\in [0,1].
\end{equation}
Subsequently, for every integer $n\geq 0$, we denote by $\mc{C}^n_0[0,1]$ the closed subspace of
the real Banach space $\mc{C}^n[0,1]$ consisting of all functions $u\in \mc{C}^n[0,1]$ such that
$u(0)=u(1)=0$, and denote $\mc{C}[0,1]:=\mc{C}^0[0,1]$, $\mc{C}_0[0,1]:=\mc{C}_0^0[0,1]$. The next result collects a  pivotal property of the integral operator ${K}$.

\begin{lemma}
\label{le2.2}
$K : \mc{C}[0,1]\to \mc{C}_0^2[0,1]$ is  linear and continuous.
\end{lemma}
\noindent \textbf{Proof:}
As the integral is linear, ${K}$ is  linear. Moreover, setting
$u:={K}f$, we have that
$$
  u'(x)=  -\int_0^x f(s)\,ds - \int_0^1 (s-1)f(s)\,ds \quad \hbox{and}\quad u''(x)=-f(x)
$$
for all $x\in [0,1]$. Thus,
$$
  \|u\|_\infty \leq 3 \|f\|_\infty,\quad   \|u'\|_\infty \leq  2 \|f\|_\infty,\quad
  \|u''\|_\infty  \leq \|f\|_\infty.
$$
Therefore, for every $f \in \mc{C}[0,1]$,
$$
  \| {K} f\|_{\mc{C}^2[0,1]}\leq 6 \|f\|_\infty,
$$
which  ends the proof.
\hfill $\Box$ \vspace{0.4cm}

Subsequently, we  consider the canonical injection
\begin{equation}
\label{2.4}
 \jmath : \mc{C}^2_0[0,1] \hookrightarrow \mc{C}^1_0[0,1].
\end{equation}
Thanks to the Ascoli--Arzel\`{a} theorem, it is a linear compact operator. Thus,
\begin{equation}
\label{2.5}
  \mc{K}  := \jmath  {K}|_{\mc{C}_0^1[0,1]} :\; \mc{C}_0^1[0,1]\to\mc{C}_0^1[0,1],
\end{equation}
also is a linear compact operator. Using $\mc{K}$, the problem \eqref{1.1} can be expressed as a fixed point equation for a compact operator, because  $u$ solves \eqref{1.1} if, and only if,
$$
  u = \mc{K} (\l u+ a u^2).
$$
Note that $R[\mc{K}]\subset \mc{C}_0^2[0,1]$, by the definition of $\mc{K}$.
Thus, the solutions of \eqref{1.1} are the zeroes of the nonlinear operator
$\mf{F}:\R\times \mc{C}_0^1[0,1]\to  \mc{C}_0^1[0,1]$ defined by
\begin{equation}
\label{2.6}
  \mf{F}(\l,u):=u-\mc{K}(\l u+a u^2),\qquad (\l,u)\in\R\times \mc{C}_0^1[0,1].
\end{equation}
Setting
\begin{equation}
\label{2.7}
  \mf{L}(\l) u  := u-\l \mc{K} u, \quad \mf{N}(\l,u):=-\mc{K} (au^2), \qquad  (\l,u)\in\R\times \mc{C}_0^1[0,1],
\end{equation}
it is apparent that
$$
   \mf{F}(\l,u)=\mf{L}(\l)u+\mf{N}(\l,u)
$$
satisfies the general structural  requirements of Chapters 2 and 6 of \cite{LG01}, because $\mf{L}(\l)$ is an analytic  compact perturbation of the identity map on $\mc{C}_0^1[0,1]$, $I$, and the nonlinearity is completely continuous, i.e., continuous and compact, and, being a polynomial, also is analytic. In particular, $\mf{L}(\l)$ is Fredholm of index zero for all $\l\in\R$, and $\mf{F}$ is a compact perturbation of the identity map that it is real analytic in $(\l,u)\in\R\times\mc{C}_0^1[0,1]$. Thus, the main theorems of Crandall and Rabinowitz \cite{CR1,CR2}, as well as the unilateral global bifurcation theorem of L\'{o}pez-G\'{o}mez \cite[Th. 6.4.3]{LG01}, can be applied.
\par
The generalized spectrum, $\Sigma(\mf{L})$, of the Fredholm curve $\mf{L}(\l)$ defined in \eqref{2.7}  consists of the set of $\l\in\R$ for which $u = \l \mc{K} u$ for some $u\in \mc{C}_0^1[0,1]$, $u\neq 0$.  Differentiating twice with respect to $x$, this fixed point equation can be equivalently expressed as
\begin{equation}
\label{2.8}
 \left\{ \begin{array}{l} -u''=\l u \quad \hbox{in}\;\; [0,1],\\
 u(0)=u(1)=0.\end{array}\right.
\end{equation}
Thus,
$$
  \Sigma(\mf{L}) =\left\{\s_n\equiv (n\pi)^2\;:\; n\in\N,\;  n\geq 1\right\}.
$$
Moreover,
$$
  N[\mf{L}(\l_n)]=\mathrm{span\,}[\psi_n],\qquad \psi_n(x)=\sin (n\pi x),\quad x\in[0,1].
$$
Subsequently, in order to apply the main theorem of \cite{CR1} at $\s_n=(n\pi)^2$, we fix
$n\geq 1$ and, adopting the notations of \cite[Ch. 2]{LG01}, we set $(\l_0,\v_0)\equiv (\s_n,\psi_n)$.
Then,
$$
  N[\mf{L}_0]=\mathrm{span\,}[\v_0],\qquad \mf{L}_0\equiv \mf{L}(\l_0)=I-\s_n \mc{K},\qquad \mf{L}_1\equiv \mf{L}'(\l_0)= -\mc{K},
$$
and the following transversality condition holds
\begin{equation}
\label{2.9}
  \mf{L}_1(N[\mf{L}_0])\oplus R[\mf{L}_0]=\mc{C}_0^1[0,1].
\end{equation}
On the contrary, assume that $\mf{L}_1 \v_0 \in R[\mf{L}_0]$. Then, there exists $u\in \mc{C}_0^1[0,1]$
such that
$$
  -\mc{K} \v_0=\mf{L}_1\v_0 = \mf{L}_0 u=u-\s_n \mc{K} u
$$
and hence, differentiating twice with respect to $x$, it becomes apparent that
$$
  -\v_0 =- u''-\s_n u.
$$
Consequently, multiplying by $\v_0$ and integrating in $(0,1)$ yields
$$
  -\int_0^1 \v_0^2\,dx = \int_0^1 [(-u''-\s_n u)\v_0]\,dx=\int_0^1 [(-\v_0''-\s_n\v_0)u]\,dx=0,
$$
which is impossible. This shows \eqref{2.9}. Therefore,   as a direct application of the main theorem of Crandall and Rabinowitz \cite{CR1}, the following result of a local nature holds.

\begin{theorem}
\label{th2.1} For any given integer $n\geq 1$, let $Y$ denote the closed
subspace of $\mc{C}_0^1[0,1]$ defined by
$$
  Y:= \Big\{ w\in \mc{C}_0^1[0,1]\;:\; \int_0^1 w(x)\psi_n(x)\,dx =0\Big\}.
$$
Then,   there exist $\eta>0$ and two analytic maps $\l_n : (-\eta,\eta)\to \R$ and $y_n : (-\eta,\eta)\to Y$ such that
\begin{itemize}
\item $\l_n(0)=\s_n$, $y_n(0)=0$;
\item $\mf{F}(\l_n(s),s(\psi_n+y_n(s)))=0$ for all $s\in(-\eta,\eta)$; and
\item  the solutions of the curve $(\l_n(s),s(\psi_n+y_n(s)))$, $|s|<\eta$,  are the unique zeroes of $\mf{F}$, besides $(\l,0)$,  in a neighborhood of  $(\s_n,0)$ in $\R\times \mc{C}_0^1[0,1]$.
\end{itemize}
\end{theorem}

Since $y_n(0)=0$,  the function $u(s)=s(\psi_n+y_n(s))\in \mc{C}_0^2[0,1]$ has the same nodal behavior as $\psi_n$ for sufficiently small $s\neq 0$, because $y_n(s)\sim 0$ in $\mc{C}_0^1[0,1]$ and the zeroes of $\psi_n$ are simple. Therefore, by Theorem \ref{th2.1}, it is apparent that, for every $n\geq 1$,  \eqref{1.1} has a curve of solutions with $n-1$ zeroes bifurcating from $u=0$ at $\l=\s_n$, regardless the nature of the weight function $a\in \mc{C}[0,1]$. In particular, by the local uniqueness result at $(\s_n,0)$, the positive solution of \eqref{1.1} can only bifurcate from $u=0$ at the critical value of the parameter $\s_1=\pi^2$. Next, we will analyze the local nature of this bifurcation from $(\l,u)=(\s_1,0)$. Setting $D_{1} :=\l_1'(0)$, $D_{2} :=\l_1''(0)$, $w_{1} := y'_1(0)$ and $w_{2} := y''_1(0)$, we have that
$$
    \l_1(s)= \s_1+s D_{1} + s^2 D_{2} +O(s^3),\qquad y_{1}(s)=sw_{1} + s^2 w_{2} +O(s^3),
$$
as $s\to 0$. By Theorem \ref{th2.1}, we already know that
$$
   \begin{aligned}
   -s (\psi_1+s w_{1} + s^2 w_{2} &+O(s^3))''  = [\s_1+s D_{1} + s^2 D_{2} +O(s^3)\\ &+a(x)s(\psi_1+s w_{1} + s^2 w_{2} +O(s^3))]s(\psi_1+s w_{1} + s^2 w_{2} +O(s^3))
   \end{aligned}
$$
for $s\simeq 0$. Thus, dividing by $s$ yields
\begin{equation}
\label{2.10}
 \begin{aligned}
 	-(\psi_1+s w_{1} + s^2 w_{2} &+O(s^3))''  = [\s_1+s D_{1} + s^2 D_{2} +O(s^3)\\ &+a(x)s(\psi_1+s w_{1} + s^2 w_{2} +O(s^3))](\psi_1+s w_{1} + s^2 w_{2} +O(s^3)).
 	\end{aligned}
\end{equation}
Particularizing \eqref{2.10} at $s=0$, yields to  $-\psi_1''=\s_1\psi_1$, which holds true
by the definition of $\psi_1$. Identifying terms of the first order in $s$, it follows from \eqref{2.10} that
\begin{equation}
\label{2.10.1}
   -w_{1}'' = \s_1 w_{1} +(D_{1}+a(x)\psi_1) \psi_1.
\end{equation}
Therefore, multiplying by $\psi_1$ this equation and integrating in $(0,1)$ yields
\begin{equation}
\label{2.11}
  D_{1} = -\frac{\int_0^1 a(x)\psi_1^3(x)\,dx}{\int_0^1 \psi_1^2(x)\,dx}=-2 \int_0^1a(x)\sin^3(\pi{x})\,dx.
\end{equation}
If $D_{1}\neq 0$, then, since $a(x)$ changes the sign, the \emph{bifurcation direction}, $D = D_1$, can take any value, either positive, or negative. Actually, the bifurcation to positive solutions is supercritical if $D>0$, while it is subcritical if $D<0$. If $D_1=0$, then we need to compute $D_{2}$.
Suppose $D_1=0$. Then, similarly as above, we can collect terms of the second order in $s$ from \eqref{2.10} to get
\begin{equation*}
-w_{2}'' = \s_1 w_{2} + D_{2}\psi_{1} + a(x)w_{1}\psi_{1}
\end{equation*}
and hence,
\begin{equation}
\label{2.11.1}
  D_{2} = -\frac{\int_0^1 a(x)w_{1}(x)\psi_1^2(x)\,dx}{\int_0^1 \psi_1^2(x)\,dx}=-2 \int_0^1 a(x)w_{1}(x)\sin^2(\pi{x})\,dx.
\end{equation}
Thus, to get the exact value of $D_{2}$, we need to determine $w_{1}(x)$. It is
the unique solution of \eqref{2.10.1}, subject to Dirichlet boundary conditions, in
the closed subspace $Y$. Since $D_1=0$, the general solution of \eqref{2.10.1} is given by
\begin{equation*}
\begin{aligned}
w_{1}(x) &= \cos(\pi x)\left(c_{1} + \frac{1}{\pi}\int_{0}^{x} a(s)\sin^3(\pi s)\; ds \right)\\
&+ \sin(\pi x)\left(c_{2} - \frac{1}{\pi}\int_{0}^{x} a(s)\sin^2 (\pi s) \cos(\pi s)\; ds \right).
\end{aligned}
\end{equation*}
As $0=w_{1}(0) = c_{1}$, after some adjustment we find that
\begin{equation*}
w_{1}(x) = c_{2}\sin(\pi x) + \frac{1}{\pi} \int_{0}^{x} a(s) \sin^2(\pi s) \sin(\pi s - \pi x)  \; ds.
\end{equation*}
To find out $c_{2}$, we recall that $y_{1}(s)\in Y$ for $s\simeq 0$. Since $Y$ is closed, this entails that
$$
   w_1 = \lim_{s\to 0} \frac{y_1(s)}{s}\in Y.
$$
Thus,
\begin{align*}
0= & \int_{0}^{1} w_{1}(x)\sin(\pi x) \; dx \\ & = c_{2} \int_{0}^{1} \sin^2(\pi x) \; dx + \int_{0}^{1} \frac{\sin(\pi x)}{\pi} \int_{0}^{x} a(s) \sin^2(\pi s) \sin(\pi s - \pi x)  \; ds \; dx
\end{align*}
and therefore,
\begin{equation*}
c_{2} = -2 \int_{0}^{1} \frac{\sin(\pi x)}{\pi} \int_{0}^{x} a(s) \sin^2(\pi s) \sin(\pi s - \pi x)  \; ds \; dx.
\end{equation*}
This way we can compute the \emph{bifurcation direction} $D = D_{2}$ when $D_{1} = 0$. This situation arises in Sections \ref{SEC:k=2}, \ref{SEC:k=3}. Should it be $D_{2}=0$, then it is necessary to use higher order terms of $\lambda_{1}(s)$ and $y_{1}(s)$.
\par
Next, we will use the Schauder formula for determining the
Leray--Schauder degree $\mathrm{Deg\,}(\mf{L}(\l),B_R)$, $R>0$,  for every $\l \in\R\setminus \Sigma(\mf{L})$, where $B_R$ stands for the open ball of radius $R$ centered at the origin in the
real Banach space $\mc{C}_0^1[0,1]$. According to it, we already know that
\begin{equation}
\label{2.12}
  \mathrm{Deg\,}(\mf{L}(\l),B_R) = (-1)^{m(\mf{L}(\l))},
\end{equation}
where $m(\mf{L}(\l))$ stands for the sum of the algebraic multiplicities of the negative
eigenvalues of $\mf{L}(\l)$. To determine $m(\mf{L}(\l))$, we will find out all the values of $\mu\in\R$ for which there exists $u\in \mc{C}_0^1[0,1]$, $u\neq 0$, such that
\begin{equation}
\label{2.13}
  \mf{L}(\l)u= u-\l \mc{K} u =\mu u.
\end{equation}
Since $\mf{L}(\l)$ is invertible for all $\l \in (0,\s_1)$, by the homotopy invariance of
the degree,
$$
  d_1 \equiv \mathrm{Deg\,}(\mf{L}(\l),B_R)\quad \hbox{is constant on}\;\;  \l\in (0,\s_1).
$$
Similarly, for every $k\geq 2$,
$$
  d_{k+1} \equiv \mathrm{Deg\,}(\mf{L}(\l),B_R)\quad \hbox{is constant on}\;\;  \l\in (\s_{k},\s_{k+1}).
$$
The equation \eqref{2.13} can be expressed as
$$
  \mc{K} u = \frac{1-\mu}{\l}u,
$$
or, equivalently, by inverting $\mc{K}$,
$$
  -u'' = \frac{\l}{1-\mu}u \quad \hbox{in}\;\; [0,1].
$$
Note that, due to \eqref{2.13},   $\mc{K}u=0$ if $\mu=1$, because $\l>0$, and hence $u=0$. Thus, $\mu\neq 1$ and hence, we can divide by $1-\mu$. Consequently, there should exist some integer $n\geq 1$ such that
$$
   \frac{\l}{1-\mu} = \s_n
$$
for some $n\geq 1$. Therefore, the set of (classical) eigenvalues of $\mf{L}(\l)$ is given by
\begin{equation}
\label{2.14}
  \s(\mf{L}(\l))=\Big\{ \mu_n := 1-\frac{\l}{\s_n}\;:\; n\in\N, \; n\geq 1\Big\}.
\end{equation}
On the other hand, for every $\l>0$ and any integer $n\geq 1$, the eigenvalue $\mu_n:= 1-\frac{\l}{\s_n}$ is an algebraically simple eigenvalue of $\mf{L}(\l)=I-\l \mc{K}$, because
\begin{equation}
\label{2.15}
  N[I-\l \mc{K}-\mu_n I]= \mathrm{span\,}[\psi_n]\quad\hbox{and}\quad
  \psi_n \notin R[I-\l \mc{K}-\mu_n I].
\end{equation}
Indeed, arguing by contradiction, assume that, for some $u\in \mc{C}_0^1[0,1]$,
$$
   \psi_n = u-\l \mc{K} u -\mu_n u = (1-\mu_n)u-\l \mc{K} u.
$$
Then,
$$
(1-\mu_n)u = \l \mc{K} u +\psi_n \in \mc{C}_0^2[0,1]
$$
and, since $\mu_n\neq 1$, $u\in \mc{C}_0^2[0,1]$. Thus, differentiating twice with respect to $x$ yields
$$
  -(1-\mu_n) u'' =\l u -\psi_n''= \l u + \s_n \psi_n.
$$
Equivalently, by definition of $\mu_n$,
$$
  -\frac{\l}{\s_n}u'' -\l u = \s_n \psi_n
$$
and hence,
$$
  - u'' -\s_n u =\frac{\s_n^2}{\l}\psi_n.
$$
Finally, multiplying by $\psi_n$ and integrating in $(0,1)$, we find that
$$
  \frac{\s_n^2}{\l}\int_0^1 \psi_n^2 = \int_0^1 [(-u''-\s_n u)\psi_n]=0,
$$
which is impossible. This ends the proof of \eqref{2.15}. Therefore, \eqref{2.12} becomes
\begin{equation}
\label{2.16}
  \mathrm{Deg\,}(\mf{L}(\l),B_R) = (-1)^{n(\l)},
\end{equation}
where $n(\l)$ stands for the number of negative eigenvalues of \eqref{2.14}. Assume $\l<\s_1$. Then, $\frac{\l}{\s_1}<1$ and, since  $\s_n\geq \s_1$ for each $n\geq 1$, we find that
$$
  1-\frac{\l}{\s_n}\geq 1-\frac{\l}{\s_1}>0.
$$
Thus, $n(\l)=0$ and \eqref{2.16} entails $\mathrm{Deg\,}(\mf{L}(\l),B_R)=1$.
Assume $\s_1<\l<\s_2$. Then,
$$
  1-\frac{\l}{\s_1} < 0 < 1-\frac{\l}{\s_2}<1-\frac{\l}{\s_3}<\cdots
$$
and hence, $n(\l)=1$. Therefore, by \eqref{2.16}, $\mathrm{Deg\,}(\mf{L}(\l),B_R)=-1$. Obviously,
every time that $\l$ crosses an additional eigenvalue $\s_n$, $n(\l)$ increases by $1$. Therefore,
for every integer $k\geq 1$,
\begin{equation}
\label{4.43}
  \mathrm{Deg\,}(\mf{L}(\l),B_R)=\left\{ \begin{array}{ll} 1 & \quad \hbox{if}\;\;
  \l \in (\s_{2k},\s_{2k+1}),\\ -1 & \quad \hbox{if} \;\; \l \in (\s_{2k+1},\s_{2k+2}). \end{array}\right.
\end{equation}
Consequently, according to \cite[Th. 6.2.1]{LG01}, the next result holds. We are denoting by $\mc{S}$ the set of non-trivial solutions of \eqref{1.1}, i.e.,
$$
  \mc{S}=\{(\l,u)\in\mf{F}^{-1}(0)\;:\; u\neq 0\}\cup (\{(\s_n,0)\;:\; n\geq 1\} \subset\R\times \mc{C}_0^1[0,1],
$$
where $\mf{F}(\l,u)$ is the operator introduced in \eqref{2.6}.

\begin{theorem}
\label{th2.2}
For every $n\geq 1$, there exists a component of $\mc{S}$, $\mathscr{C}_n$, such that $(\s_n,0)\in \mathscr{C}_n$. Moreover, for sufficiently small $\e>0$,
$$
  B_\e(\s_n,0)\cap \mathscr{C}_n =\{(\l_n(s),s(\psi_n+y_n(s)))\;:\; s\sim 0\},
$$
where $(\l_n(s),s(\psi_n+y_n(s)))$, $s\sim 0$, is the analytic curve given by Theorem \ref{th2.1}.
\end{theorem}

As the nodes of the solutions of \eqref{1.1} are simple, it is easily seen that
the number of nodes of the solutions of \eqref{1.1} vary continuously in $\R\times \mc{C}_0^1[0,1]$
and hence, since the solutions of $\mathscr{C}_n$ bifurcating from
$u=0$ at $\l =\s_n$ possess $n-1$ interior nodes in $(0,1)$, $\mathscr{C}_n\setminus\{(\s_n,0)\}$
consists of solutions with $n-1$ interior nodes. Therefore,
$$
  \mathscr{C}_n\cap \mathscr{C}_m=\emptyset \quad \hbox{if}\;\; n\neq m.
$$
Consequently, thanks to the global alternative of Rabinowitz \cite{Ra1},
$\mathscr{C}_n$ is  unbounded in $\R\times \mc{C}_0^1[0,1]$ for each integer $n\geq 1$.
Note that $\mathscr{C}^+$ is the subcomponent of $\mathscr{C}_1$ consisting of the positive solutions $(\l,u)\in\mathscr{C}_1$. According to \cite[Th. 6.4.3]{LG01}, $\mathscr{C}^+$ also is unbounded in
$\R\times \mc{C}_0^1[0,1]$. A further rather standard compactness argument, whose details are omitted here, shows that actually  $\mathscr{C}^+$ is unbounded in $\R\times \mc{C}_0[0,1]$. This information can be summarized into the next result.

\begin{theorem}
\label{th2.3}
The component $\mathscr{C}^+$ is unbounded in $\R \times \mc{C}_0[0,1]$. Moreover,
$\l=\pi^2$ if $(\l,0)\in \mathscr{\bar C}^+$. Furthermore, $\mathscr{C}^+$ bifurcates supercritically from $u=0$ at $\l=\pi^2$  if $D>0$, while it does it subcritically if $D<0$, where $D$ is given by \eqref{2.11} or \eqref{2.11.1}.
\end{theorem}

In addition, by the a priori bounds of Amann and L\'{o}pez-G\'{o}mez \cite{ALG}, the next result holds.

\begin{theorem}
\label{th2.4}
The component $\mathscr{C}^+$ is uniformly bounded on any compact subinterval of $\l\in\R$, and
\eqref{1.1} cannot admit a positive solution for sufficiently large $\l$. Thus,
$$
    (-\infty,\pi^2)\subset \mathcal{P}_\l (\mathscr{C}^+).
$$
Moreover, if \eqref{1.1} admits a positive solution, $(\l_0,u_0)$, with
$\l_0>\pi^2$, then it admits at least two positive solutions for every $\l \in (\pi^2,\l_0)$.
\end{theorem}

These findings can be complemented with the theory of G\'{o}mez-Re\~{n}asco and L\'{o}pez-G\'{o}mez \cite{GRLGJDE,GRLGDIE}, later refined in \cite{LG16}, up to characterize the existence of linearly stable positive solutions of \eqref{1.1} thorough the sign of $D$. Indeed, by \cite[Cor. 9.10]{LG16}, any positive
solution of \eqref{1.1} is linearly unstable if $D\leq 0$, and actually, due to \cite[Pr. 9.2]{LG16}, \eqref{1.1} cannot admit a positive solution $(\l,u)$ with $\l \geq \pi^2$ in such case. Thus,
$$
  \mathcal{P}_\l (\mathscr{C}^+)=(-\infty,\pi^2)
$$
if $D\leq 0$. Moreover, \eqref{1.1} admits some stable positive solution if, and only if,
$D>0$ and, in such case, the results of \cite[Sec. 9.2-4]{LG16}, provide us with the next one.

\begin{theorem}
\label{th2.5}
Assume $D>0$. Then, there exists $\l_t>\pi^2$ such that \eqref{1.1} cannot admit a positive solution
if $\l > \l_t$, and
$$
  \mathcal{P}_\l (\mathscr{C}^+)=(-\infty,\l_t].
$$
Moreover,
\begin{enumerate}
\item[{\rm (a)}] Any positive solution of \eqref{1.1} with $\l\leq \pi^2$ is linearly unstable.
\item[{\rm (b)}] For every $\l \in (\pi^2,\l_t]$, the minimal positive solution of \eqref{1.1},
$(\l,\t^\mathrm{min}_\l)$, is the unique stable positive solution of \eqref{1.1}.
Moreover, these solutions are linearly stable if $\l\in(\pi^2,\l_t)$. Thus, they are local exponential attractors of \eqref{1.2}.
\item[{\rm (c)}]  For every $\l \in (\pi^2,\l_t)$, \eqref{1.1} possesses, at least, two positive solutions: one linearly stable and another one unstable.
\item[{\rm (d)}] $(\l_t,\t^\mathrm{min}_{\l_t})$ is the unique positive solution of \eqref{1.1} at
$\l=\l_t$, and it is linearly neutrally stable. Moreover, the set of positive solutions of \eqref{1.1} in a neighborhood of $(\l_t,\t^\mathrm{min}_{\l_t})$ consists of a quadratic subcritical turning point whose lower half-curve is filled in by
linearly stable positive solutions, while its upper half-curve consists of unstable solutions
with one-dimensional unstable manifold.
\item[{\rm (e)}] The map $\l\to \t^\mathrm{min}_{\l}$, $(\s_1,\l_t)\to \mc{C}_0^1[0,1]$, is analytic and
$$
  \lim_{\l\ua \l_t} \t^\mathrm{min}_{\l} = \t^\mathrm{min}_{\l_t}.
$$
\end{enumerate}
\end{theorem}

The numerical experiments carried out in this paper confirm and illuminate
these  findings complementing them. Note that the \emph{exchange stability principle} of Crandall and Rabinowitz \cite{CR2} only provides us with the linearized stability of the minimal positive solution for $\l>\pi^2$ sufficiently close to $\pi^2$, while the existence and the uniqueness of the stable positive solution established by Theorem \ref{th2.5} inherits a global character. Very recently, it
has been established by Fern\'{a}ndez-Rinc\'{o}n and L\'{o}pez-G\'{o}mez \cite{FRLG} that choosing a a nonlinearity of the type $u^p$ for some  $p\geq 2$ in  \eqref{1.1} is imperative for the validity of Theorem \ref{th2.5}, regardless the nature of the boundary conditions that might be of general type. This explains why in this paper we are focusing attention on the particular example \eqref{1.1}.

\section{Behavior of the solutions of \eqref{1.1} and \eqref{1.2} as $\l\da -\infty$}

\noindent The next result provides us with the point-wise behavior of the positive solutions of
\eqref{1.1} in the open set $\O_-$.

\begin{theorem}
\label{th3.1}
For every  $\l < \pi^2$, let $u_\l$ be a positive solution of \eqref{1.1}. Then,
\begin{equation}
\label{3.1}
  \lim_{\l\da -\infty}u_\l (x)=0 \quad \hbox{for all}\;\; x\in \O_-
\end{equation}
uniformly on compact subintervals of $\O_-$.
\end{theorem}
\begin{proof}
Pick an arbitrary $x_0\in \O_-$. As $\O_-$ is open, there exists $\e>0$ such that
$$
   0<x_0-4 \e <x_0+4\e<1 \quad\hbox{and}\quad [x_0-4\e,x_0+4\e]\subset \O_-.
$$
Since $a$ is continuous, we have that
$$
  \o := \max_{|x-x_0|\leq 4\e}a(x) <0.
$$
Let $\ell^\mathrm{min}_{\l}$ denote the minimal positive solution of the singular problem
\begin{equation}
\label{3.2}
		\left\{\begin{array}{l}
			-\ell'' = \lambda \ell + a(x)\ell^{2}\quad\text{in }\;\; (x_0-4\e,x_0+4\e),\\
			\ell(x_0-4\e) = \ell(x_0+4\e) = \infty,\end{array}\right.
\end{equation}	
whose existence is guaranteed by, e.g., \cite[Ch. 3]{LG16}, and set
$$
   B := \| \ell^\mathrm{min}_{\l}\|_{\mc{C}[x_0-2\e,x_0+2\e]}.
$$
Then, the restriction of the function $\ell^\mathrm{min}_{\l}$ to the interval
$[x_0-2\e,x_0+2\e]$ provides us with a positive subsolution of the regular problem
\begin{equation}
\label{3.3}
		\left\{\begin{array}{l}
			- u '' = \lambda u  + \o u^{2}\quad\text{in }\;\; (x_0-2\e,x_0+2\e),\\
			u(x_0-2\e) = u(x_0+2\e) = B. \end{array}\right.
\end{equation}	
As $\o<0$, any sufficiently large constant, $M>B$, provides us with a  supersolution of
\eqref{3.3} such that
$$
  \ell^\mathrm{min}_{\l} < M \quad \hbox{in}\;\; [x_0-2\e,x_0+2\e].
$$
Thus, thanks to, e.g.,  \cite[Th. 2.4]{LG16}, \eqref{3.3} possesses a unique positive solution,
$\t_{[\l,B]}$, such that
$$
   \ell^\mathrm{min}_{\l} \leq  \t_{[\l,B]} \leq  M \quad \hbox{in}\;\; [x_0-2\e,x_0+2\e].
$$
Moreover, according to the proof of \cite[Th 3.2]{LG16},  $\t_{[\l,B]}$ is symmetric about
$x_0$, and, for every $\l \in\R$, the point-wise  limit
$$
   L_{\l} := \lim_{\xi\ua\infty} \t_{[\l,\xi]}
$$
is increasing and, thanks to the uniqueness result of \cite{LG07},  it provides us with the unique positive solution of the singular problem
\begin{equation*}
		\left\{\begin{array}{l}
			- u '' = \lambda u  + \o u^{2}\quad\text{in }\;\; (x_0-2\e,x_0+2\e),\\
			u(x_0-2\e) = u(x_0+2\e) = \infty. \end{array}\right.
\end{equation*}	
Since $L_\l$ is symmetric about $x_0$, it is apparent that
\begin{equation}
\label{3.4}
  L_\l(x_0)=\min_{(x_0-2\e,x_0+2\e)} L_\l.
\end{equation}
On the other hand, by \cite[Th. 2.4]{LG16}, we already know that $\t_{[\l,\xi]}< \t_{[\mu,\xi]}$ if
$\l<\mu$. Thus, letting $\xi\ua\infty$ yields
$$
  L_\l \leq L_\mu \quad\hbox{in}\;\; [x_0-2\e,x_0+2\e].
$$
Subsequently, we consider the auxiliary function
$$
  \v(x):= \sin\frac{\pi (x-x_0+\e)}{2\e},\qquad x\in [x_0-\e,x_0+\e].
$$
It has been chosen to satisfy
\begin{equation}
\label{3.5}
		\left\{\begin{array}{l}
		-\varphi''  = \left( \tfrac{\pi}{2\e} \right)^2 \varphi  \quad\text{ in }\;\; (x_0-\e,x_0+\e),\\
		\varphi(x_0-\e)  = \varphi(x_0+\e) = 0. \end{array}\right.
\end{equation}	
Thus, multiplying the differential equation
$$
     -L_\l'' =\l L_\l +\o L_\l^2\quad \hbox{in}\;\; [x_0-\e,x_0+\e]
$$
by $\v$ and integrating in  $(x_0-\e,x_0+\e)$ yields
\begin{equation}
\label{3.6}
  -\int_{x_0-\e}^{x_0+\e} L_\l''\v\,dx = \l \int_{x_0-\e}^{x_0+\e} L_\l \v \,dx
   +\o \int_{x_0-\e}^{x_0+\e} L_\l^2\v \,dx.
\end{equation}
On the other hand, integrating by parts, we find that
\begin{align*}
  \int_{x_0-\e}^{x_0+\e} L_\l''\v\,dx & =
  \int_{x_0-\e}^{x_0+\e} \left( L_\l'\v\right)'\,dx -
  \int_{x_0-\e}^{x_0+\e} L_\l'\v'\,dx=
  -\int_{x_0-\e}^{x_0+\e} L_\l'\v'\,dx,\\
  \int_{x_0-\e}^{x_0+\e} L_\l\v''\,dx & =
  \int_{x_0-\e}^{x_0+\e} \left( L_\l\v'\right)'\,dx -
  \int_{x_0-\e}^{x_0+\e} L_\l'\v'\,dx.
\end{align*}
Consequently, by \eqref{3.5},
\begin{align*}
-\int_{x_0-\e}^{x_0+\e} L_\l''\v\,dx & = \int_{x_0-\e}^{x_0+\e}L_\l'\v'\,dx =
\int_{x_0-\e}^{x_0+\e} \left( L_\l\v'\right)'\,dx - \int_{x_0-\e}^{x_0+\e}L_\l\v''\,dx\\
& = L_\l(x_0+\e)\v'(x_0+\e)-L_\l(x_0-\e)\v'(x_0-\e) +  \left( \tfrac{\pi}{2\e} \right)^2
\int_{x_0-\e}^{x_0+\e} L_\l\v\,dx.
\end{align*}
Thus, since $\o<0$, substituting in \eqref{3.6} yields
\begin{align*}
 \Big[ \left( \tfrac{\pi}{2\e} \right)^2-\l\Big]  \int_{x_0-\e}^{x_0+\e} L_\l\v\,dx & =
 \o \int_{x_0-\e}^{x_0+\e} L_\l^2\v \,dx+L_\l(x_0-\e)\v'(x_0-\e) -L_\l(x_0+\e)\v'(x_0+\e)
 \\ & < L_\l(x_0-\e)\v'(x_0-\e) -L_\l(x_0+\e)\v'(x_0+\e).
\end{align*}
Therefore, since
$$
  L_\l(x_0-\e)\v'(x_0-\e) -L_\l(x_0+\e)\v'(x_0+\e)>0,
$$
we can infer from the previous estimate that
\begin{equation}
\label{3.7}
     \lim_{\l\da-\infty} \int_{x_0-\e}^{x_0+\e} L_\l\v\,dx =0.
\end{equation}
Consequently, owing to \eqref{3.4} and \eqref{3.7}, it becomes apparent that
\begin{equation}
\label{3.8}
     \lim_{\l\da-\infty} L_\l(x_0) =0.
\end{equation}
Note that, since
$$
   \ell^\mathrm{min}_{\l} \leq  \t_{[\l,B]} \leq L_\l \quad \hbox{in}\;\; (x_0-2\e,x_0+2\e),
$$
\eqref{3.8} implies that
\begin{equation}
\label{3.9}
     \lim_{\l\da-\infty} \ell^\mathrm{min}_{\l}(x_0) =0.
\end{equation}
Similarly, for every $x\in [x_0-\e,x_0+\e]$, we have that
$$
   [x-\e,x+\e]\subset [x_0-2\e,x_0+2\e]
$$
and  hence, the restriction of $\ell^\mathrm{min}_{\l}$ to the interval $[x-\e,x+\e]$ provides us with a subsolution of
\begin{equation}
\label{3.10}
		\left\{\begin{array}{l}
			- u '' = \lambda u  + \o u^{2}\quad\text{in }\;\; (x-\e,x+\e),\\
			u(x-\e) = u(x+\e) = B. \end{array}\right.
\end{equation}	
Consequently, reasoning as above, it becomes apparent that
\begin{equation}
\label{3.11}
  \ell^\mathrm{min}_{\l} \leq L_{\l,x}\quad \hbox{in}\;\; (x-\e,x+\e),
\end{equation}
where $L_{\l,x}$ stands for the unique positive solution of the singular problem
\begin{equation*}
		\left\{\begin{array}{l}
			-u'' = \lambda u + \o u^{2}\quad\text{in }\;\; (x-\e,x+\e),\\
			u(x-\e) = u(x+\e) = \infty.\end{array}\right.
\end{equation*}	
By the uniqueness and the radial symmetry of $L_{\l,x}$ about $x$, we find that
$$
  L_{\l,x}(y)= L_\l (x_0-x+y) \quad\hbox{for all}\;\; y\in (x-\e,x+\e).
$$
Thus, it follows from \eqref{3.11} that
$$
  \ell^\mathrm{min}_{\l} (x) \leq L_{\l,x}(x)=L_\l(x_0) \quad \hbox{for all}\;\; x\in(x_0-\e,x_0+\e).
$$
Therefore, due to \eqref{3.8}, we find that
$$
   \lim_{\l\da -\infty}  \ell^\mathrm{min}_{\l} =0 \quad \hbox{uniformly in} \;\; (x_0-\e,x_0+\e).
$$
A compactness argument ends the proof.
\end{proof}

Throughout the rest of this section, we will assume that $a(x)$ satisfies ($\mathrm{H}_a$). Then, for every $j\in\{1,...,r\}$, some of the following (excluding) options occurs. Either  (i) $0<\a_j<\b_j<1$, or (ii) $0=\a_j<\b_j<1$, or (iii) $0<\a_j<\b_j=1$. Subsequently,
we will denote by $\ell^\mathrm{min}_{\l,j}$ the minimal positive solution of the singular problem
\begin{equation}
\label{iii.12}
 \left\{ \begin{array}{l} -u''=\l u +a(x)u^2  \quad \hbox{in}\;\; (\a_j,\b_j),\cr
 u(\a_j)=\infty, \;\; u(\b_j)=\infty, \end{array}\right.
\end{equation}
if (i) holds, or the minimal positive solution of
\begin{equation}
\label{iii.13}
 \left\{ \begin{array}{l} -u''=\l u +a(x)u^2  \quad \hbox{in}\;\; (\a_j,\b_j),\cr
 u(0)=0, \;\; u(\b_j)=\infty, \end{array}\right.
\end{equation}
if (ii) holds, or the minimal positive solution of
\begin{equation}
\label{iii.14}
 \left\{ \begin{array}{l} -u''=\l u +a(x)u^2  \quad \hbox{in}\;\; (\a_j,\b_j),\cr
 u(\a_j)=\infty, \;\; u(1)=0, \end{array}\right.
\end{equation}
in case (iii). Their existence is guaranteed, e.g., by \cite[Ch. 3]{LG16}.
\par
The proof of Theorem \ref{th3.1} reveals that actually the next result holds.

\begin{corollary}
\label{co3.1}
Under assumption {\rm ($\mathrm{H}_a$)}, for every $j\in \{1,...,r\}$ and $x\in (\a_j,\b_j)$,
\begin{equation}
\label{iii.15}
  \lim_{\l\da -\infty}\ell^\mathrm{min}_{\l,j}(x)=0,
\end{equation}
uniformly in compact subsets of $(\a_j,\b_j)$.
\end{corollary}

Conversely, Theorem \ref{th3.1}  follows from Corollary \ref{co3.1} taking into account that, thanks to the maximum principle,
\begin{equation}
\label{iii.16}
  u_\l \leq \ell^\mathrm{min}_{\l,j} \qquad \hbox{in}\;\; (\a_j,\b_j)
\end{equation}
for all $j\in\{1,...,r\}$.
\par
The behavior of the positive solutions of \eqref{1.1} as $\l\da -\infty$ in $\O_-$ described by
Theorem \ref{th3.1} is mimicked by the positive solutions of its parabolic counterpart
\eqref{1.2}, as soon as the initial datum $u_0$ be a subsolution of \eqref{1.1}. To state this result, we need to introduce some of notation. We will denote by $u(x,t;u_0,\l)$ the unique solution of \eqref{1.2}, and by $T_\mathrm{max}=T_\mathrm{max}(u_0,\l)\in (0,+\infty]$ its maximal existence time. As for every $\l<\mu$ the solution $u(x,t;u_0,\mu)$ is a strict supersolution of \eqref{1.2}, owing to the parabolic maximum principle,
\begin{equation}
\label{iii.17}
  u(x,t;u_0,\l)<u(x,t;u_0,\mu)
\end{equation}
for all $x\in(0,1)$ and $t \in [0,T_\mathrm{max}(u_0,\mu))$. Thus,
\begin{equation}
\label{iii.18}
   T_\mathrm{max}(u_0,\mu)\leq T_\mathrm{max}(u_0,\l) \quad \hbox{for all}\;\;\l<\mu.
\end{equation}
Therefore, the limit
\begin{equation}
\label{iii.19}
  T_\mathrm{max}(u_0)\equiv \lim_{\l\da -\infty} T_\mathrm{max}(u_0,\l)\in (0,\infty]
\end{equation}
is well defined.

\begin{theorem}
\label{th3.2} Suppose that $u_0\gneq 0$ is a subsolution of \eqref{1.1}. Then, for every $x\in\O_-$ and
$t \in (0,T_\mathrm{max}(u_0))$,
\begin{equation}
\label{iii.20}
  \lim_{\l\da -\infty} u(x,t;u_0,\l)=0.
\end{equation}
Moreover, the limit is uniform on compact subsets of $\O_-$.
\end{theorem}
\begin{proof}
Pick $t\in  (0,T_\mathrm{max}(u_0))$. By \eqref{iii.18} and \eqref{iii.19}, there exists $\mu<0$ such that $t\in (0,T_\mathrm{max}(u_0,\mu))$ for all $\l<\mu$. Moreover, since
$u_0$ is a subsolution of \eqref{1.1}, by Sattinger \cite{Sa}, $u(x,t;u_0,\l)$ is a subsolution of \eqref{1.1} for all $t\in [0,T_\mathrm{max}(u_0,\l))$; equivalently, $u(x,t;u_0,\l)$ is non-decreasing for all $t\in [0,T_\mathrm{max}(u_0,\l))$. In particular, for every $j\in\{1,...,r\}$, the restriction of $u(\cdot,t;u_0,\l)$ to the interval $I^-_j=(\a_j,\b_j)$ provides us with a positive
subsolution of the singular boundary value problem \eqref{iii.12} if (i) holds, \eqref{iii.13} if (ii) holds, or \eqref{iii.14} if (iii) holds. Thus, by the maximum principle, we find that
\begin{equation}
\label{iii.21}
  u(x,t;u_0,\l)\leq \ell^\mathrm{min}_{\l,j} \quad \hbox{in}\;\; [\a_j,\b_j].
\end{equation}
Finally,  \eqref{iii.20} follows easily from \eqref{iii.21} and Corollary \ref{co3.1}.
\end{proof}

Since \eqref{iii.20} holds for every $t\in (0,T_\mathrm{max}(u_0))$, letting $t\da 0$ in \eqref{iii.20}, after inter-exchanging the two limits,  it seems that a necessary condition so that $u_0\gneq 0$ can be a subsolution of \eqref{1.1} as $\l\da -\infty$ should be
$$
  u_0\equiv 0 \quad \hbox{in}\;\; \O_-=\cup_{j=1}^r I_j^-,
$$
which explains why all subsolutions of \eqref{1.1} that we will consider later
satisfy this requirement.
\par
Finally, we are ready to deliver the proof of Theorem \ref{th1.1}. We will actually prove that
one can take $T(u_0)=T_\mathrm{max}(u_0)$.
\vspace{0.2cm}

\noindent \emph{Proof of Theorem \ref{th1.1}.} First, suppose that $I^+_i$ is an interior interval. Note that $I^+_i= (\g_i,\r_i)=(\b_j,\a_{j+1})$. Choose $\e$ sufficiently small so that $\a_j+\e<\b_{j+1}-\e$. Then, by Corollary \ref{co3.1},
\begin{equation}
\label{iii.23}
  \lim_{\l\da -\infty}\ell^\mathrm{min}_{\l,j}(\a_j+\e)=
  \lim_{\l\da -\infty}\ell^\mathrm{min}_{\l,j+1}(\b_{j+1}-\e)=0.
\end{equation}
Moreover, by \eqref{iii.16}, for every $t\in (0,T_\mathrm{max}(u_0))$, there is $\mu<0$
such that $t\in (0,T_\mathrm{max}(u_0,\mu))$ and, for each $\l<\mu$,
\begin{equation}
\label{iii.24}
  u(x,t;u_0,\l)\leq \ell^\mathrm{min}_{\l,h}(x) \quad \hbox{for all}\;\;
  x\in (\a_h,\b_h), \;\; h\in\{j,j+1\}.
\end{equation}
Thus, by \eqref{iii.23} and \eqref{iii.24},
\begin{equation}
\label{iii.25}
  \lim_{\l \da -\infty}u(\a_j+\e,t;u_0,\l)= \lim_{\l \da -\infty}u(\b_{j+1}-\e,t;u_0,\l)=0.
\end{equation}
On the other hand, note that, as soon as the condition
\begin{equation}
\label{iii.26}
  \l +au(\cdot,t;u_0,\l)\leq 0 \quad \hbox{in}\;\; [\a_j+\e,\b_{j+1}-\e]
\end{equation}
holds, we have that
$$
  \frac{\p u}{\p t}=\frac{\p^2 u}{\p x^2}+\l u+ au^2 \leq \frac{\p^2 u}{\p x^2}
  \quad \hbox{in}\;\; [\a_j+\e,\b_{j+1}-\e].
$$
Since $u_0=0$ in $[\a_j,\b_{j+1}]$, by continuity, there exists $T(\l)>0$ such that \eqref{iii.26} holds for all $t\in [0,T(\l)]$. Actually, by \eqref{iii.17}, \eqref{iii.26} holds
for every $t\in [0,T(\mu)]$ if $\l<\mu$. Consequently, setting $T\equiv T(\mu)$ and
$$
  Q_T\equiv (\a_j+\e,\b_{j+1}-\e)\times [0,T],
$$
the parabolic maximum principle implies that
$$
  \max_{\bar Q_T}u = \max_{\p_L Q_T}u,
$$
where $\p_L Q_T$ stands for the \emph{lateral boundary} of the parabolic cylinder $Q_T$,
$$
  \p_L Q_T \equiv (\{\a_j+\e,\b_{j+1}-\e\}\times [0,T])\cup ([\a_j+\e,\b_{j+1}-\e]\times \{0\}).
$$
Therefore, since $u_0=0$ in $[\a_j,\b_{j+1}]$,
\begin{equation}
\label{iii.27}
\begin{split}
  \max_{\bar Q_T} u & =\max_{t\in [0,T]}\max \{u(\a_j+\e,t;u_0,\l),u(\b_{j+1}-\e,t;u_0,\l)\}
  \\ & = \max\{u(\a_j+\e,T;u_0,\l),u(\b_{j+1}-\e,T;u_0,\l)\}
\end{split}
\end{equation}
because, since it is a subsolution of \eqref{1.1},  $u(x,t;u_0,\l)$ is increasing in time.
\par
Subsequently, we choose $\eta>0$ arbitrary. Then, there exists $\mu=\mu(\eta)<0$ such that, for every $\l<\mu$,
\begin{equation*}
  \l + \|a\|_\infty \eta < \frac{\l}{2}<0.
\end{equation*}
Thanks to \eqref{iii.25}, shortening $\mu$ if necessary, we also have that, for every $\l<\mu$,
$$
  \max \{u(\a_j+\e,T;u_0,\l),u(\b_{j+1}-\e,T;u_0,\l)\}<\eta.
$$
Thus, \eqref{iii.27} implies that
$$
   \max_{\bar Q_T} u <\eta,
$$
and hence,
$$
  \l+a(x)u(x,T;u_0,\l)<\l+\|a\|_\infty \eta < \frac{\l}{2}<0 \quad \hbox{in}\;\; [\a_j+\e,\b_{j+1}-\e].
$$
Therefore, $u$ must remain bellow the level $\eta $ in  $[\a_j+\e,\b_{j+1}-\e]$
for all time where it is defined. As $\eta>0$ is arbitrary, the proof is completed in this case.
\par
This proof can be easily adapted to cover the case when $0\in \bar I_1^+$, or $1\in \bar I_s^+$. Indeed, suppose, for instance, that $I_1^+=(0,\r_1)$. Then, $I_1^-=(\r_1,\b_1)$ and, whenever $u_0=0$ in $[0,\b_1]$, the parabolic maximum principle can be applied in the interval $[0,\b_1-\e]$, instead of in $[\a_j+\e,\b_{j+1}-\e]$, to get the same result. This ends the proof. \hfill $\Box$
\vspace{0.2cm}

In the rest of this section, we assume that $a(x)$ satisfies (H$_a$) with $s=n+1$ and $r=n$, like
the special choice \eqref{1.3} with $n\in\N$, $n\geq 1$. Suppose
$$
  I_i^+=(\g_i,\r_i),\qquad i\in \{1,...,n+1\},
$$
and, for every $i\in \{1,...,n+1\}$, let $\t_{\{\l,i\}}$ be a positive solution of
\begin{equation}
\label{iii.28}
		\left\{\begin{aligned}
			-u'' &= \lambda u + a^+(x)u^{2}\quad\text{in }(\g_i,\r_i),\\
			u(\g_i) &= u(\r_i) = 0.
		\end{aligned}\right.
\end{equation}
on the component $\mathscr{C}^+$ of this problem. The existence has been already discussed in Section 2 and follows from the a priori bounds of \cite{ALG}. The  uniqueness
might be an open problem even in the special case when
\begin{equation}
\label{iii.29}
   a^+(x)=\mu_i \sin \frac{\pi(x-\g_i)}{\r_i-\g_i},\qquad x\in [\g_i,\r_i],
\end{equation}
for some constant $\mu_i>0$. As, for every $\lambda<0$,  the change of variable $u\equiv -\lambda U$ transforms \eqref{iii.28} in
\begin{equation}
\label{iii.30}
		\left\{\begin{aligned}
			-\varepsilon U'' &= -U + a^+(x)U^{2} \quad\text{in }(\g_i,\r_i),\\
			U(\g_i) &= U(\r_i) = 0,
		\end{aligned}\right.
\end{equation}	
with $\varepsilon = -1/\lambda$, it turns out that the problem of the uniqueness of the positive solution of \eqref{iii.28} as $\lambda \da -\infty$ is equivalent to the problem of the uniqueness of the positive solution for the singular perturbation problem \eqref{iii.30} as $\varepsilon\da 0$. Although there is a huge amount of literature on multi-peak solutions for Schr\"{o}dinger type equations
like \eqref{iii.30} (see, e.g., Ambrosetti, Badiale and Cingolani \cite{ABC}, del Pino and Felmer \cite{DPFa,DPFb},  Dancer and Wei \cite{DW}, and Wei \cite{Wei}), the experts still seem to be focusing most of their efforts into the problem of the existence of multi-bump solutions, rather than on the problem of their uniqueness (see, e.g., the recent paper of Le, Wei and Xu \cite{LWX}).
\par
Our numerical experiments suggest that the problem \eqref{iii.28}, with the special choice \eqref{iii.29}, possesses a unique positive solution, $\theta_{\lambda,i}$,  in the component  $\mathscr{C}^+$ for every $\lambda<\pi^2$. Moreover, $\theta_{\lambda,i}$ is symmetric about the central point of $(\g_i,\r_i)$, $z_i:=(\g_i+\r_i)/2$, where $a^+(x)$ reaches its maximum value in $I_i^+$, and it has a single peak at $z_i$. Actually,
$$
\mathscr{C}^+=\Big\{(\l,\t_{\l,i})\;:\;\; \l < \left(\tfrac{\pi}{\r_i-\g_i}\right)^2\Big\}
$$
consists of symmetric solutions about $z_j$, because we could not find any secondary bifurcation point
along the curve $\mathscr{C}^+$. Figure \ref{Newbif_diag} shows the global bifurcation diagram of positive solutions of \eqref{iii.28} for the  choice \eqref{iii.29}, with $\mu_i=1$, after re-scaling the problem to the entire interval $[0,1]$. We  are plotting the parameter $\l$, in abscisas, versus the derivative of the solution at the origin, $u'(0)$, in ordinates. As for $\lambda <-600$, $\theta'_{\lambda,i}(0)$ is very small, in this range of values of $\l$ it is hard to differentiate $\mathscr{C}^+$ from the $\lambda$-axis. The component $\mathscr{C}^+$
bifurcates subcritically from $u=0$ at $\l=\pi^2$ and, according to
\cite{ALG}, satisfies $\mathcal{P}_\lambda(\mathscr{C}^+)=(-\infty,\pi^2)$. It consists of symmetric solutions about $0.5$ with a single peak at $0.5$.

\begin{figure}[h!]		
\centering
\includegraphics[scale = 0.6]{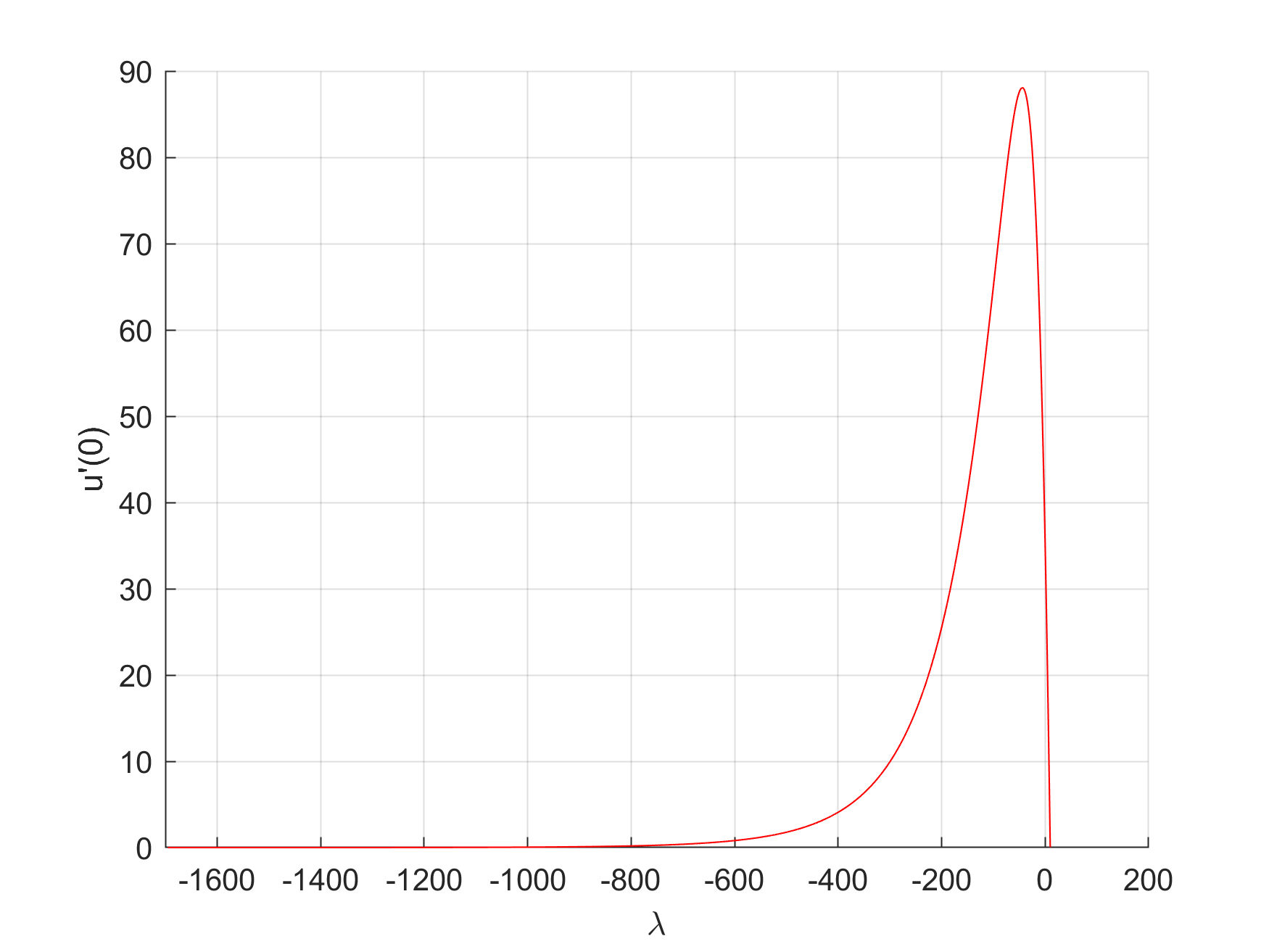}
\caption{Global bifurcation diagram of \eqref{iii.28}--\eqref{iii.29}.}\label{Newbif_diag}
\end{figure}

Figure  \ref{sinpeak} shows the plots of the solutions of $\mathscr{C}^+$ corresponding to
$\l=-100$, $\l=-683$ and $\l=-1695$, respectively. Not surprisingly, the smaller is the value of $\lambda$, the more concentrated is the mass of $\theta_{\lambda,i}$ at $0.5$. The three solutions plotted in Figure \ref{sinpeak} have been previously re-scaled to the interval $[0,1]$ from the original interval $[\g_i,\r_i]$ in such a way that $0.5$ corresponds to $z_i$. Actually,  according  to our numerical experiments, it becomes apparent that, for every $x \in [\g_i,\r_i]$,
$$
  \lim_{\l\da-\infty}\theta_{\l,i}(x)=\left\{\begin{array}{ll} +\infty & \quad \hbox{if}\;\; x=z_i,\\
  0 & \quad \hbox{if}\;\; x\neq z_i,\end{array}\right.
$$
which is a rather genuine behavior of this type of superlinear elliptic boundary value problems.

\begin{figure}[h!]		
\centering
\begin{subfigure}[t]{0.32\textwidth}
\centering
\includegraphics[scale = 0.32]{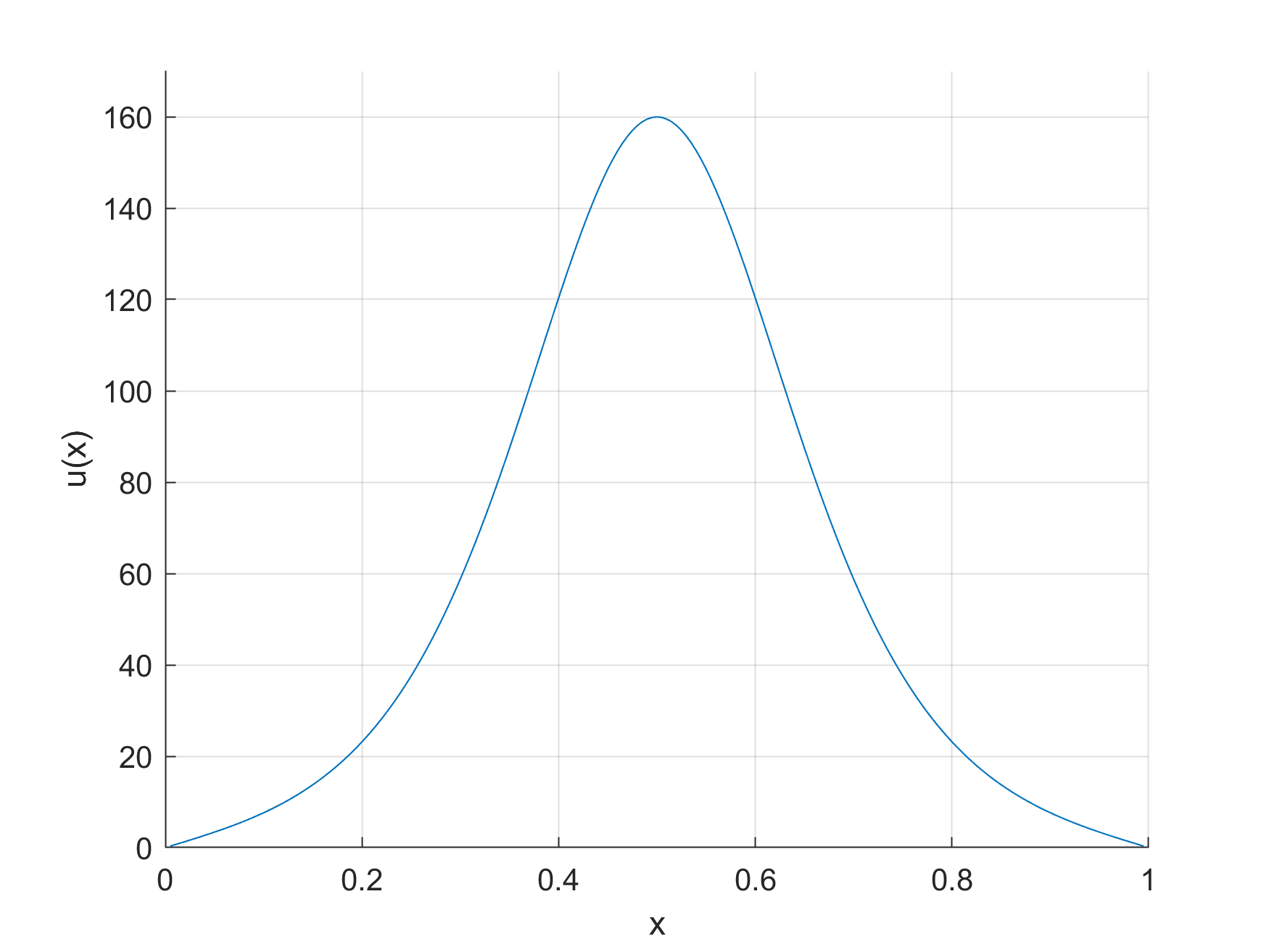}
\caption{$\lambda=-100$}
\end{subfigure}%
\begin{subfigure}[t]{0.32\textwidth}
\centering
\includegraphics[scale = 0.32]{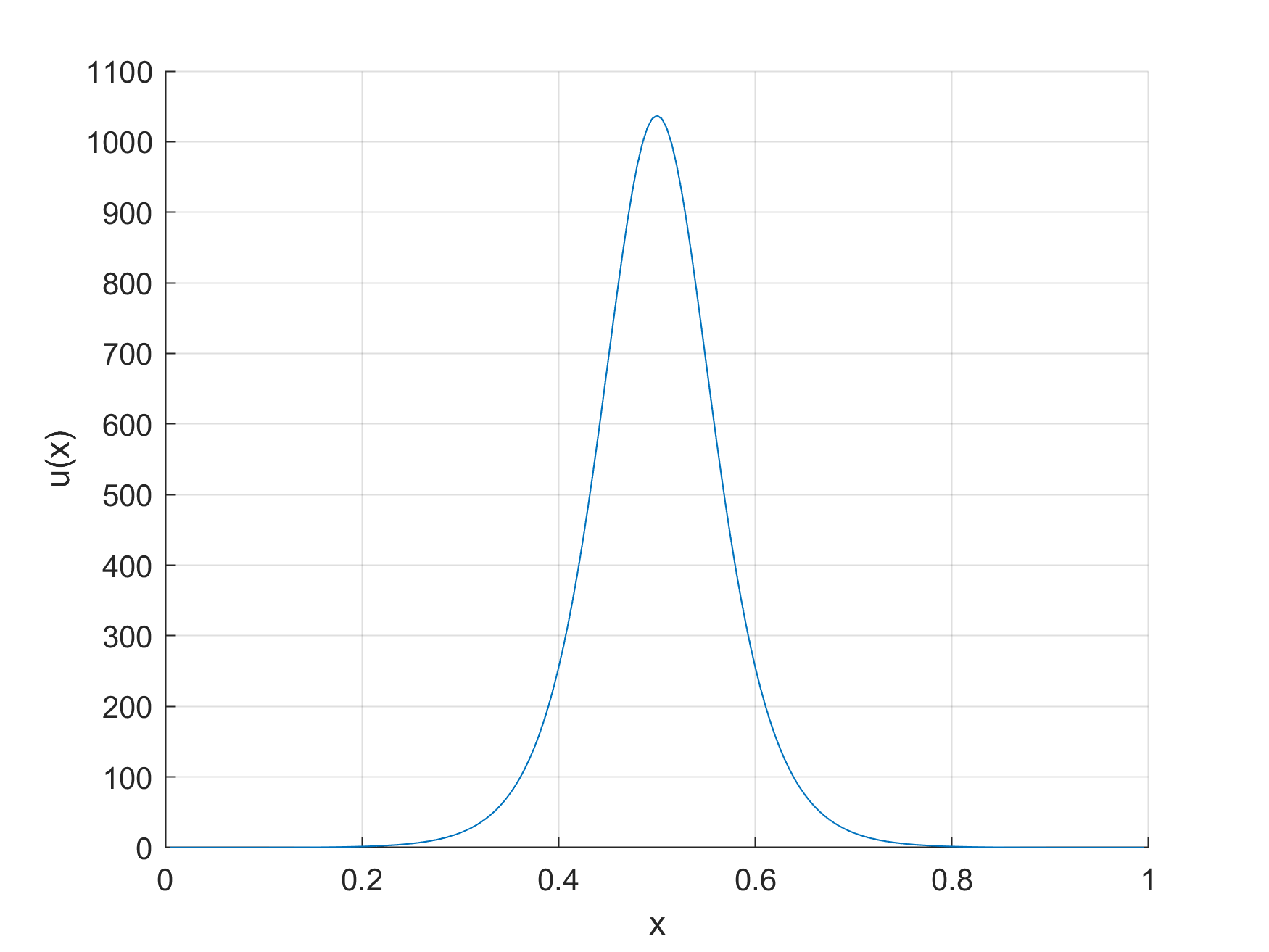}
\caption{$\lambda=-683$}
\end{subfigure}	
\begin{subfigure}[t]{0.32\textwidth}
\centering
\includegraphics[scale = 0.32]{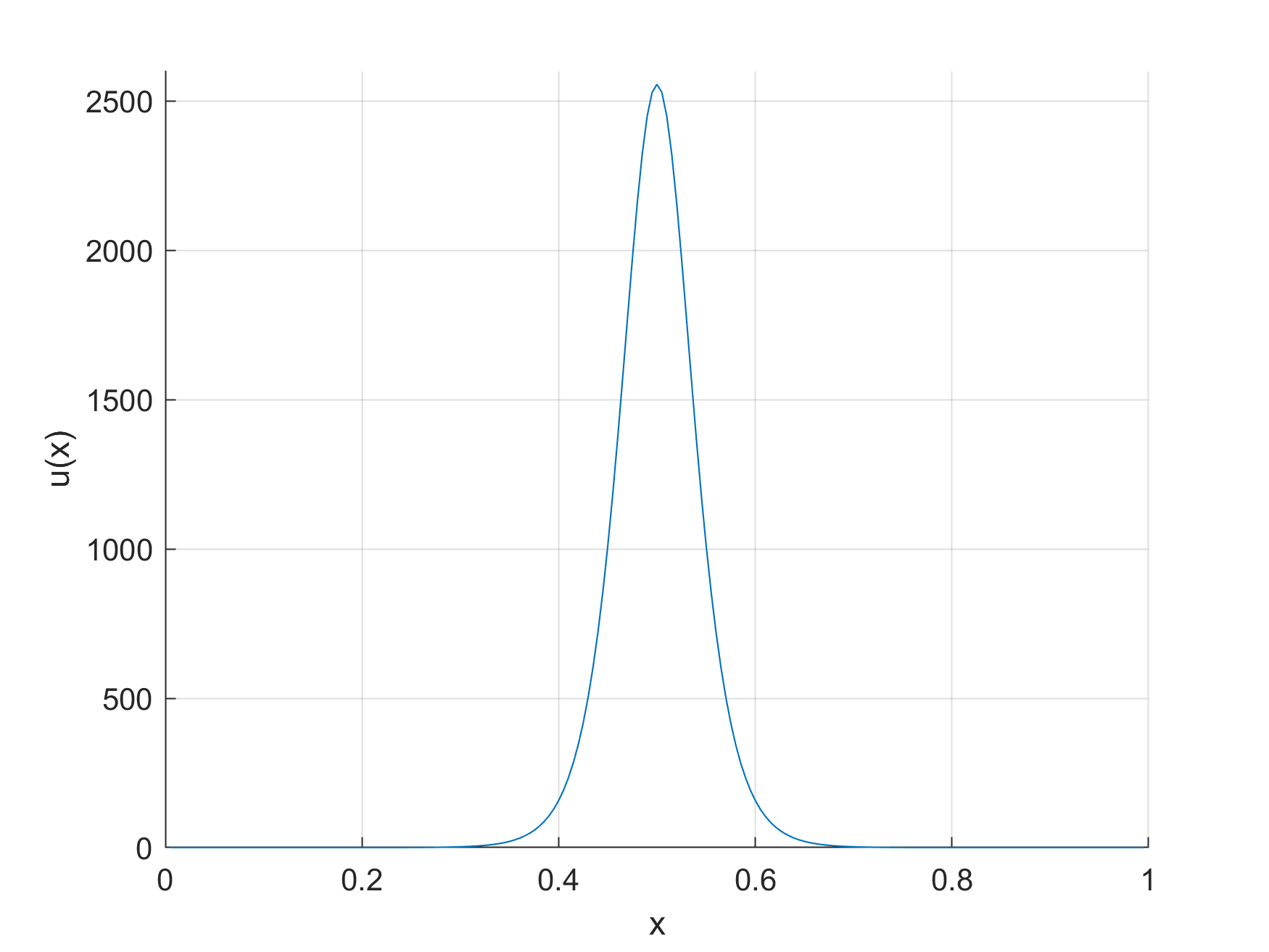}
\caption{$\lambda=-1695$}
\end{subfigure}	
\caption{Three positive solutions of the component $\mathscr{C}^+$.}
\label{sinpeak}
\end{figure}

Subsequently, we will consider the subsolution of \eqref{1.1} defined by
\begin{equation}
\label{iii.31}
   \underline{u}_{_{\underbrace{1...1}_{n+1}} }:= \left\{ \begin{array}{ll}
  \t_{\l,i}&\quad \hbox{in}\;\; [\g_i,\r_i], \;\; i\in \{1,...,n+1\},\\[2ex]
  0 &\quad \hbox{in}\;\; [0,1] \setminus \bigcup_{i=1}^{n+1} [\g_i,\r_i].\end{array} \right.
\end{equation}
The fact that it provides us with a weak subsolution of \eqref{1.1} is a direct consequence of a result of Berestycki and Lions \cite{BL}. Thus, making the choice $u_0:= \underline{u}_{1...1}$ and using the theory of Sattinger \cite{Sa}, $u(x,t;u_0,\l)$ must be a subsolution of \eqref{1.1} for all $t \in I_\mathrm{max}(u_0,\l)=[0,T_\mathrm{max}(u_0,\l))$. Equivalently, $u(x,t;u_0,\l)$ is non-decreasing for all $t\in I_\mathrm{max}(u_0,\l)$. In particular,
\begin{equation}
\label{iii.32}
  u_0 \leq u(\cdot,t;u_0,\l) \;\; \hbox{in}\;\; [0,1] \;\; \hbox{for all}\;\; t\in I_\mathrm{max}(u_0,\l).
\end{equation}
Then, the following result holds, though it remains an open problem to ascertain whether, or not,
the condition \eqref{iii.33} holds.

\begin{theorem}
\label{th3.3} Suppose {\rm ($\mathrm{H}_a$)} with $s=n+1$ and $r=n$, and $u_0\equiv \underline{u}_{1...1}$. Assume, in addition, that there exists $\mu>0$
such that, for every $\l<\mu$,  $T_\mathrm{max}(u_0,\l)=\infty$ and there is a constant $C(\l)>0$ such that
\begin{equation}
\label{iii.33}
  u(x,t;u_0,\l)\leq C(\l) \;\; \hbox{for all}\;\; (x,t)\in [0,1]\times [0,\infty).
\end{equation}
Then, there exists $\l_c<0$ such that \eqref{1.1} has, for every $\l<\l_c$, $2^{n+1}-1$ positive solutions.
\end{theorem}
\begin{proof}
Under these assumptions, thanks to the main theorem of Langlais and Phillips \cite{LP}, for every $\l<\mu$ the point-wise limit
$$
  \t_{\{\l,(1,...,1)\}}:=\lim_{t\ua \infty}u(\cdot,t;u_0,\l)
$$
provides us with a positive solution of \eqref{1.1} such that $u_0\leq \t_{\{\l,(1,...,1)\}}$. Thus,
for every $i\in \{1,...,n+1\}$,
$$
  \t_{\l,i}\leq \t_{\{\l,(1,...,1)\}} \quad \hbox{in}\;\; I^+_i=(\g_i,\r_i),
$$
while, thanks to Theorem \ref{th3.1},
$$
  \lim_{\l\da -\infty} \t_{\{\l,(1,...,1)\}} =0 \quad \hbox{in}\;\; \bigcup_{j=1}^nI^-_j.
$$
In particular, for sufficiently negative $\l<0$,  the positive solution $\t_{\{\l,(1,...,1)\}}$ has,
at least, one peak on each of the $n+1$ intervals $(\g_i,\r_i)$, $i\in\{1,...,n+1\}$.
\par
Now, for every $(d_1,\ldots,d_{n+1})\in \{0,1\}^{n+1}$ such that $\sum_{i=1}^{n+1}d_i\leq n$,
we consider the initial data
$$
  \tilde u_0:= \underline{u}_{(d_1,...,d_{n+1})}\equiv \left\{ \begin{array}{lll} d_i \t_{\{\l,i\}} &\quad \hbox{in}\;\; I_i^+, & \qquad
  i\in\{1,...,n+1\}, \\[1ex] 0 &\quad \hbox{in}\;\; I_j^-, & \qquad
  j\in\{1,...,n\}, \end{array}\right.
$$
having, at least, $\sum_{i=1}^{n+1}d_i\leq |n|$ peaks. Arguing as before, it becomes apparent that $u(x,t;\tilde u_0,\l)$ is a subsolution
of \eqref{1.1} for all $t \in I_\mathrm{max}(\tilde u_0,\l)$. Moreover, by the parabolic maximum principle, since $\tilde u_0\leq u_0$, we have that, for every $x\in [0,1]$ and $t\in I_\mathrm{max}(\tilde u_0,\l)$,
$$
  u(x,t;\tilde u_0,\l)\leq u(x,t;u_0,\l).
$$
Thus, $T_\mathrm{max}(\tilde u_0,\l)=\infty$ and, for sufficiently negative $\l$,
$ u(x,t;\tilde u_0,\l)$ is increasing in time and bounded above. Hence,
$$
   \t_{\{\l,(d_1,...,d_{n+1})\}} \equiv \lim_{t\ua\infty}u(\cdot,t;\tilde u_0,\l)\leq \lim_{t\ua\infty}u(\cdot,t;u_0,\l) \equiv \t_{\{\l,(1,...,1)\}}
$$
provides us with a positive solution of \eqref{1.1} such that, according to Theorems \ref{th1.1}
and \ref{th3.1},
$$
  \lim_{\l\da -\infty} \t_{\{\l,(d_1,...,d_{n+1})\}} =0 \quad \hbox{in}\;\;
  \bigcup_{j=1}^n I^-_j \cup \bigcup_{i\in Z} I^+_i
$$
where
$$
 Z\equiv \{i\in\{1,...,n+1\}:d_i=0\}.
$$
Since,
$$
  0<\t_{\{\l,i\}}\leq \t_{\{\l,(d_1,...,d_{n+1})\}}\quad\hbox{in}\;\; I^+_i \;\;
  \hbox{for all}\;\; i\in \{1,...,n+1\}\setminus N,
$$
a genuine combinatorial argument ends the proof, as  these solutions differ as $\l\da -\infty$.
\end{proof}

\section{The case $n=1$}

\noindent Throughout this section, we assume that
$$
  a(x) = \sin(3\pi x),\qquad x\in [0,1].
$$
Then, the global bifurcation diagram of the positive solutions of \eqref{1.1} looks like shows Figure \ref{FIG:sin3 diag}. Our numerical experiments suggest that the set of positive solutions of \eqref{1.1} consists of the component $\mathscr{C}^+$, which  bifurcates supercritically from $u=0$ at $\l=\pi^2$,  because
$$
  D_{1}=-2 \int_0^1a(x)\sin^3(\pi{x})\,dx=-2 \int_0^1\sin (3\pi {x}) \sin^3(\pi{x})\,dx = \frac{1}{4}>0.
$$

\begin{figure}[h!]		
\centering
\begin{subfigure}[t]{0.47\textwidth}
\centering
\includegraphics[scale = 0.5]{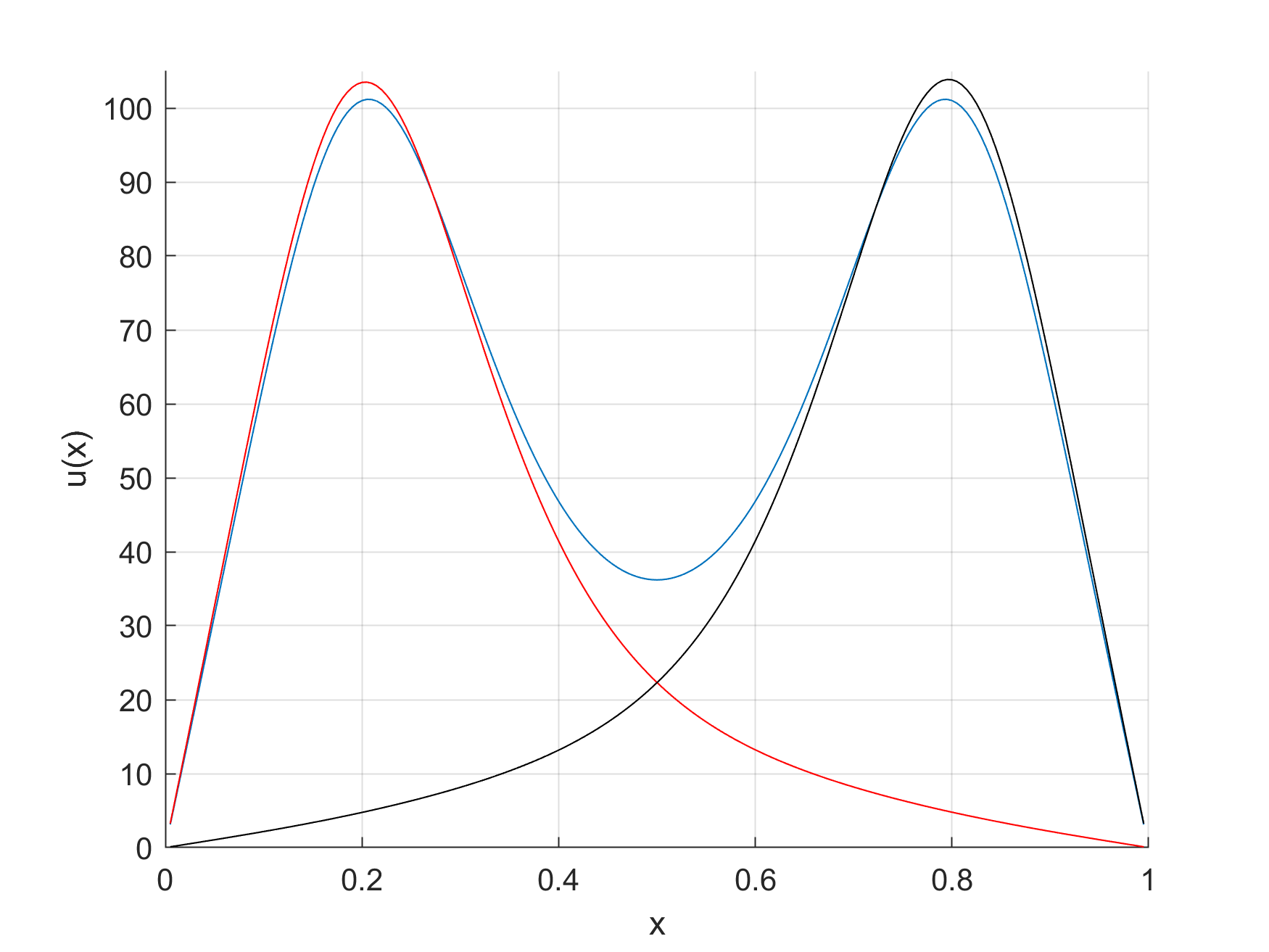}
\caption{Positive solutions for $\lambda \approx -21$: $10(1)$ in red,
$01(1)$ in black, and $11(2)$ in blue.}\label{FIG:sin3 sol examples}
\end{subfigure}%
\begin{subfigure}[t]{0.47\textwidth}
\centering
\includegraphics[scale = 0.5]{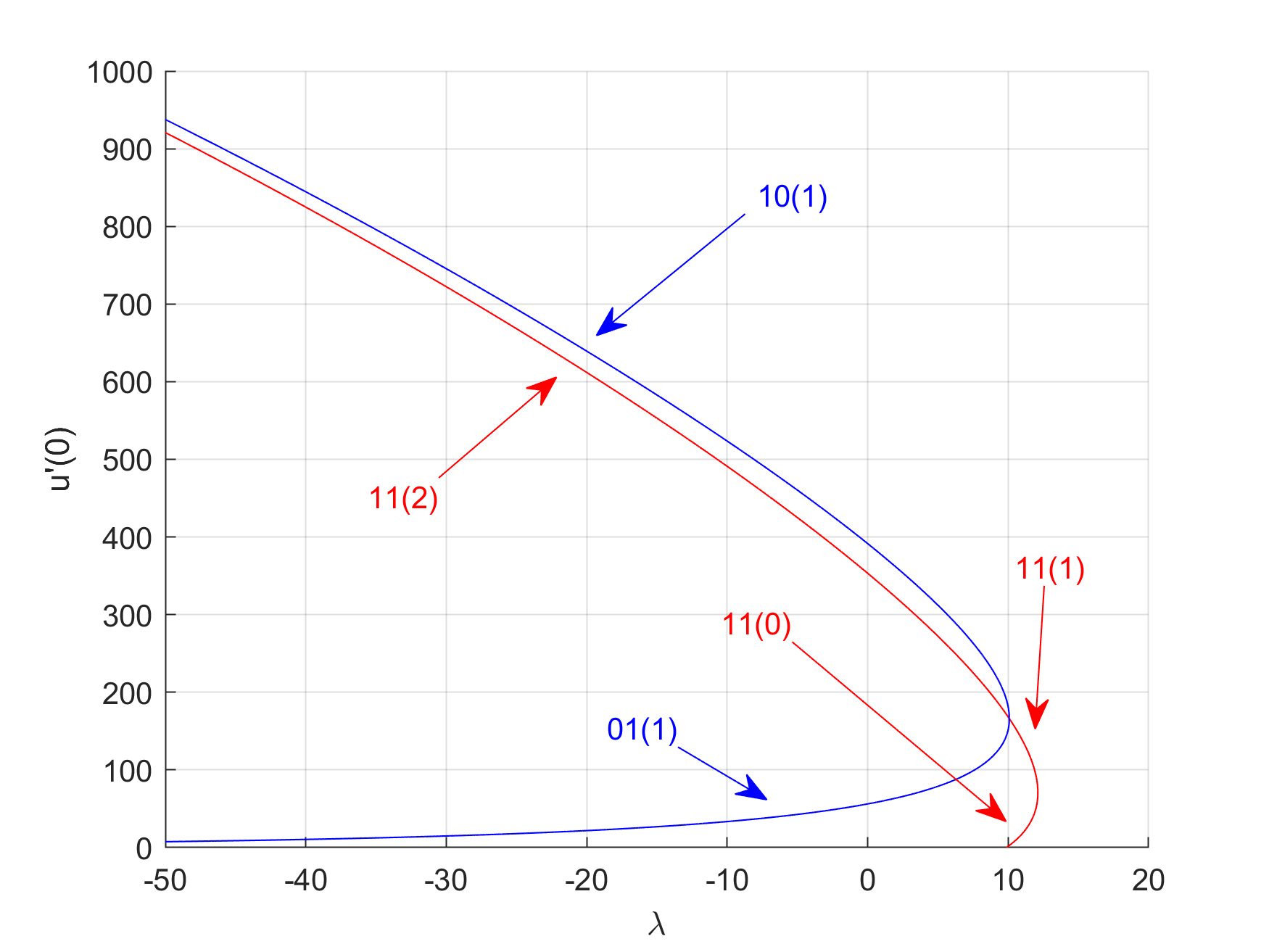}
\caption{Global bifurcation diagram.}\label{FIG:sin3 diag}
\end{subfigure}	
\caption{Numerical results for $a(x) = \sin(3\pi x)$.}
\end{figure}

This component exhibits a turning point at $\lambda_t \approx 12.1$, and a secondary bifurcation at $\lambda_s\approx 10.1$, as shown in the global bifurcation diagram plotted
in Figure \ref{FIG:sin3 diag}. In this and in all subsequent global bifurcation diagrams we are plotting the parameter $\l$, in abscisas, versus the derivative of the solution at the origin, $u'(0)$, in ordinates. This allows to differentiate between all admissible positive solutions. By the symmetries of the problem, the reflection about $0.5$ of any positive solution of \eqref{1.1} provides with another solution, though there is no way to differentiate between such solutions if, instead of plotting $\l$ versus $u'(0)$,  we plot $\l$ versus the $L_p$-norm of the solutions for some $p\geq 1$. Should we proceed in this way, we could not differentiate between, e.g.,  the solutions of types $01(1)$ and $10(1)$, as they have the same $L_p$-norms for all $p\geq 1$ (see the plots of these solutions in Figure \ref{FIG:sin3 sol examples}).

According to Theorem \ref{th2.5}, $\mc{P}_\l(\mathscr{C}^+)=(-\infty,\l_t]$ and, for every $\l\in (\pi^2,\l_t]$, the minimal positive solution of \eqref{1.1} is the unique stable positive
solution of \eqref{1.1}. Actually, for every
$\l \in (\pi^2,\l_t)$, the minimal solution is linearly asymptotically stable and hence, its Morse index equals zero. Moreover, by Theorem \ref{th2.5},  $(\l_t,\t^\mathrm{min}_{\l_t})$
is the unique positive solution at $\l_t$, it is linearly neutrally stable, and it is a
quadratic subcritical turning point of the component $\mathscr{C}^+$.
The solutions on the upper half curve through the subcritical turning point  $(\l_t,\t^\mathrm{min}_{\l_t})$ have one-dimensional unstable manifold, and actually are of type $11(1)$  until $\l$ reaches the secondary bifurcation point, $\l_s$, where they became unstable with Morse index two and type $11(2)$ for any further smaller value of $\l$.
\par
At $\l=\l_s$,  two (new) secondary branches of positive solutions with respective types $01(1)$ and $10(1)$ bifurcate subcritically. Naturally, $u'(0) \approx 0$ for the solutions of type $01(1)$, while $u'(0)$ is large for those  of type $10(1)$, as confirmed  by our numerical experiments. These three branches seem to be globally defined for all further smaller values of $\l$, $\l<\l_s$.
\par
In full agreement with Conjecture \ref{con3.1}, the problem \eqref{1.1} has three positive solutions for every $\l <\l_s$. Figure \ref{FIG:sin3 sol examples} shows the plots of these solutions at a value $\l \approx -21$. Note that the number of peaks of the solutions coincides with the dimension of their respective unstable manifolds for all $\l<\l_s$.

\section{The case $n=2$}\label{SEC:k=2}

\noindent Throughout this section we have chosen
$$
  a(x) = \sin(5\pi x),\qquad x\in [0,1].
$$
By Conjecture \ref{con3.1}, we expect to have $2^3-1=7$ positive solution for sufficiently negative $\lambda$. The global bifurcation diagram computed in this case has been plotted in Figure \ref{FIG:sin5 complete diag}.
\begin{figure}[h!]
\centering
\includegraphics[scale=0.7]{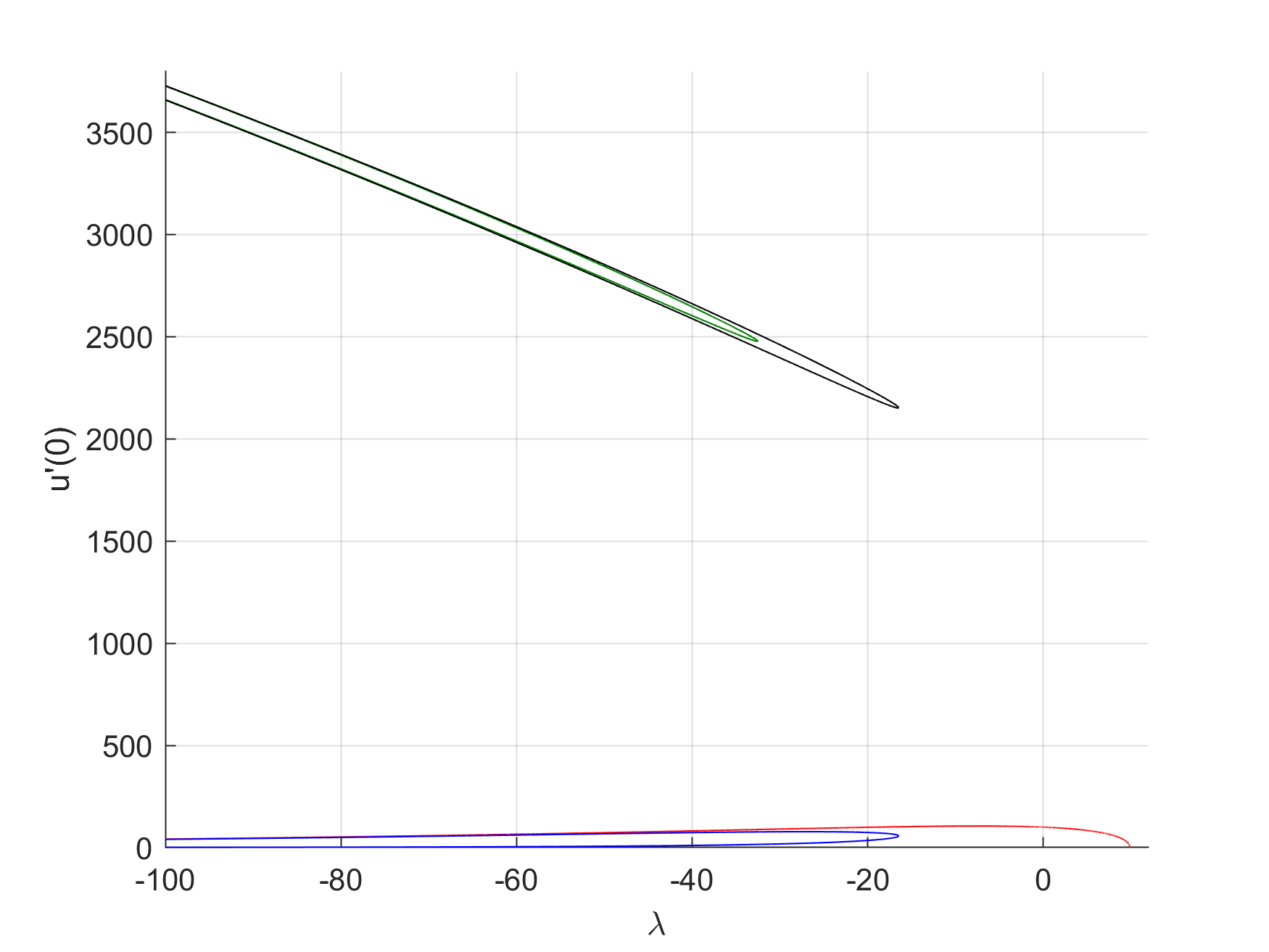}
\caption{Global bifurcation diagram for $a(x) = \sin(5\pi x)$.}\label{FIG:sin5 complete diag}
\end{figure}

It consists of $4$ components, $3$ global folds isolated from $u=0$, plus $\mathscr{C}^+$, which in this occasion bifurcates subcritically from $u=0$ at $\l=\pi^2$, because
$$
  D_{1}= 2 \int_0^1\sin (5\pi {x}) \sin^3(\pi{x})\,dx = 0
$$
and
$$
    D_{2}= -2 \int_0^1 w_{1}(x)\sin (5\pi {x}) \sin^2(\pi{x})\,dx = -\frac{5}{256 \pi^2} <0 .
$$
None of these components, neither $\mathscr{C}^+$ nor any of the three folds plotted in Figure \ref{FIG:sin5 complete diag}, exhibited any secondary bifurcation along it.
\par
Figure \ref{FIG:sin5 complete diag zoom} shows two magnifications of the most significant pieces of
the global bifurcation diagram plotted in Figure \ref{FIG:sin5 complete diag} together with
the superimposed types of the solutions along each of the solution curves plotted on it.
Precisely,  Figure \ref{FIG:sin5 top} shows a zoom of the two superior global folds plotted in Figure \ref{FIG:sin5 complete diag} around their respective turning points. These solutions look larger in these global bifurcation diagram because
$$
  \lim_{\l\da -\infty}u'_\l(0)=+\infty
$$
for any positive solution $(\l,u_\l)$ having mass in $(0,0.2)$.
Figure \ref{FIG:sin5 top} shows the types of the positive solutions along each of the half-branches
of the two folds. They change  from type $100(1)$ to type $110(2)$ as  they cross the turning
point of the exterior component, while they are changing from type $101(2)$ to type $111(3)$ as
the turning point of the interior folding is crossed.

\begin{figure}[h!]		
\centering
\begin{subfigure}[t]{0.47\textwidth}
\centering
\includegraphics[scale = 0.5]{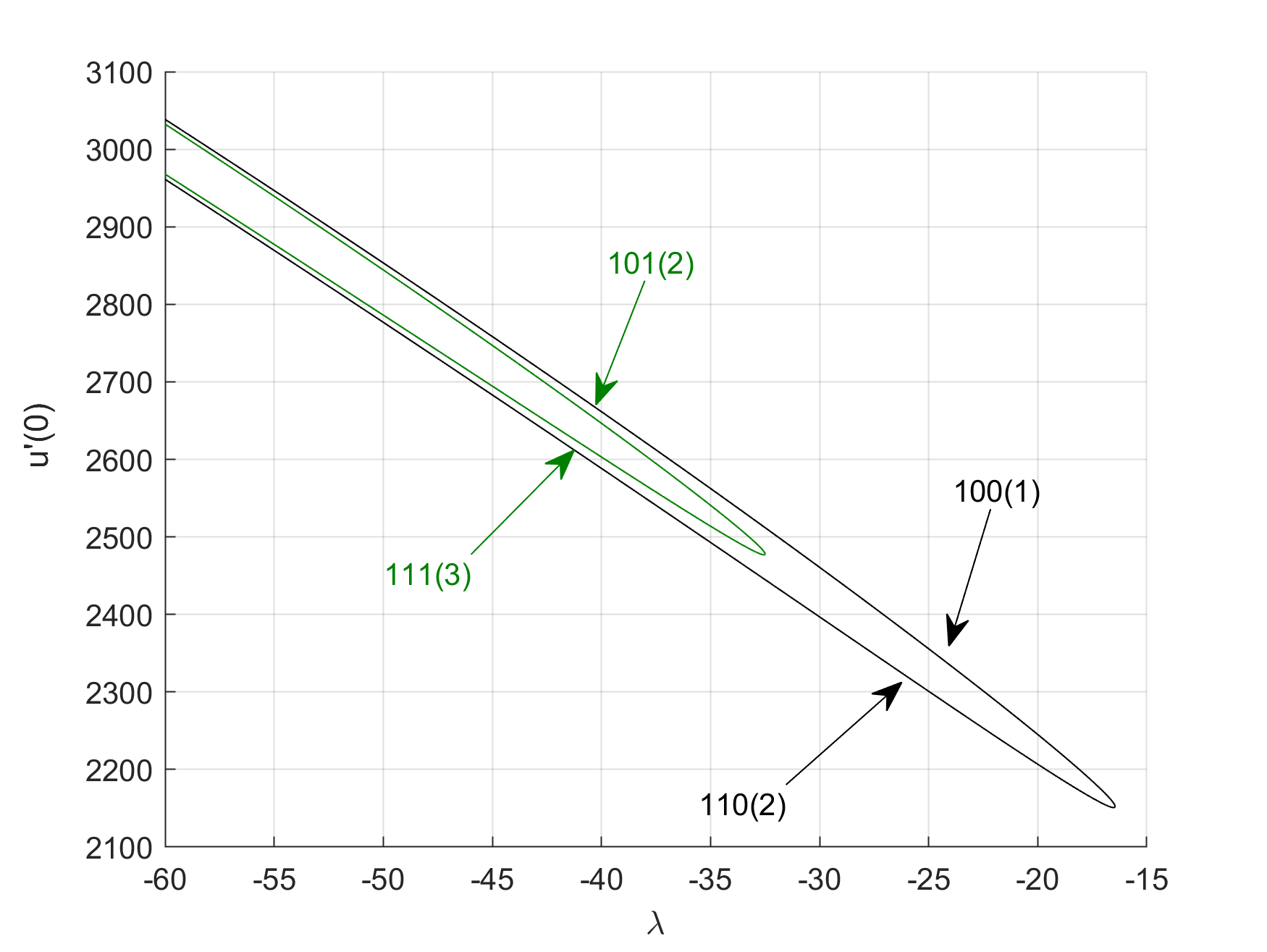}
\caption{Upper part magnification.}\label{FIG:sin5 top}
\end{subfigure}%
\begin{subfigure}[t]{0.47\textwidth}
\centering
\includegraphics[scale = 0.5]{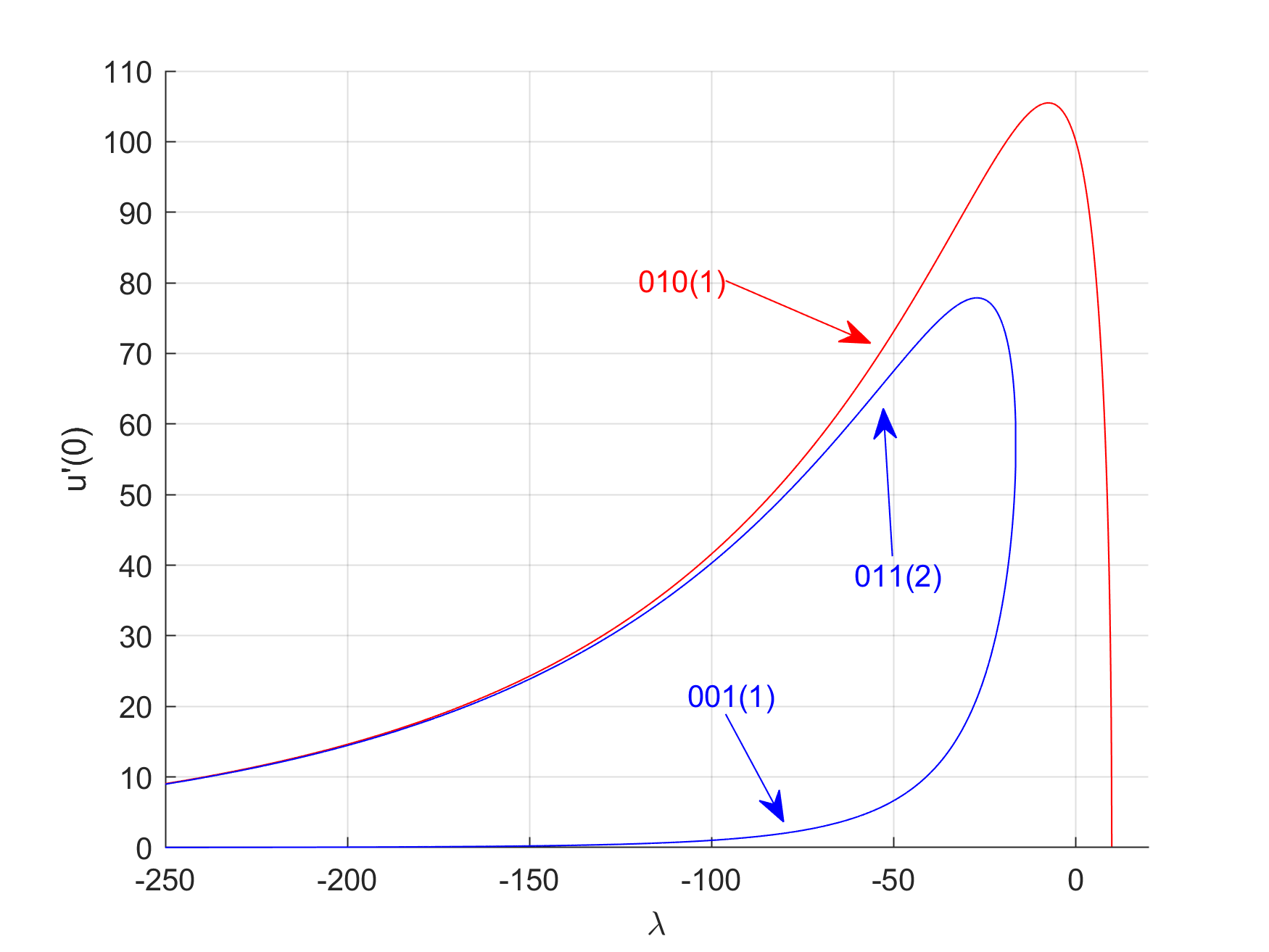}
\caption{Lower part magnification.}\label{FIG:sin5 bottom}
\end{subfigure}	
\caption{Two significant magnifications of Figure \ref{FIG:sin5 complete diag}.}
\label{FIG:sin5 complete diag zoom}
\end{figure}
	
Not surprisingly, since $\mathscr{C}^+$ does not exhibit any secondary bifurcation along it, all the solutions of $\mathscr{C}^+$ that we have computed are of type $010(1)$, in complete agreement with the exchange stability principle of Crandall and Rabinowitz \cite{CR2}, because $u=0$ is linearly stable for all $\l<\pi^2$.
\par
Lastly, the solutions along the interior folding in Figure \ref{FIG:sin5 bottom} change type from
$001(1)$ to $011(2)$ when the turning point of this components is switched on.
Moreover, for sufficiently negative $\l$, \eqref{1.1} admits 7 positive solutions, with respective types $$
   001(1), \qquad 010(1), \qquad   100(1), \qquad  101(2), \qquad  110(2), \qquad  011(2), \qquad  111(3),
$$
in full agreement with Conjecture \ref{con3.1}. In particular, in any circumstances, the number of peaks of these solutions, when they exist, equals their respective Morse indices.

\section{The case $n=3$}\label{SEC:k=3}
\noindent Throughout this section we make the choice
$$
  a(x) = \sin(7\pi x),\qquad x\in [0,1].
$$
According to Conjecture \ref{con3.1}, we expect
to have $2^4-1=15$ positive solution for sufficiently negative $\lambda$. Since
$$
   D_{1} = -2 \int_0^1 \sin(7\pi x)\sin^3(\pi x)\,dx = 0
$$
and
$$
  D_{2} = -2 \int_0^1 w_{1}(x)\sin (7\pi {x}) \sin^2(\pi{x})\,dx = \frac{1}{128 \pi^2}>0,
$$
the component $\mathscr{C}^+$ bifurcates supercritically from $u=0$ at $\l=\pi^2$, and exhibits a secondary bifurcation at $\lambda_s\approx -2.85$. It has been plotted
in Figure \ref{FIG:sin7 bottom}, which shows a significant piece of the global bifurcation diagram
of the positive solutions of \eqref{1.1}.

\begin{figure}[h!]		
\centering
\begin{subfigure}[t]{0.47\textwidth}
\centering
\includegraphics[scale = 0.5]{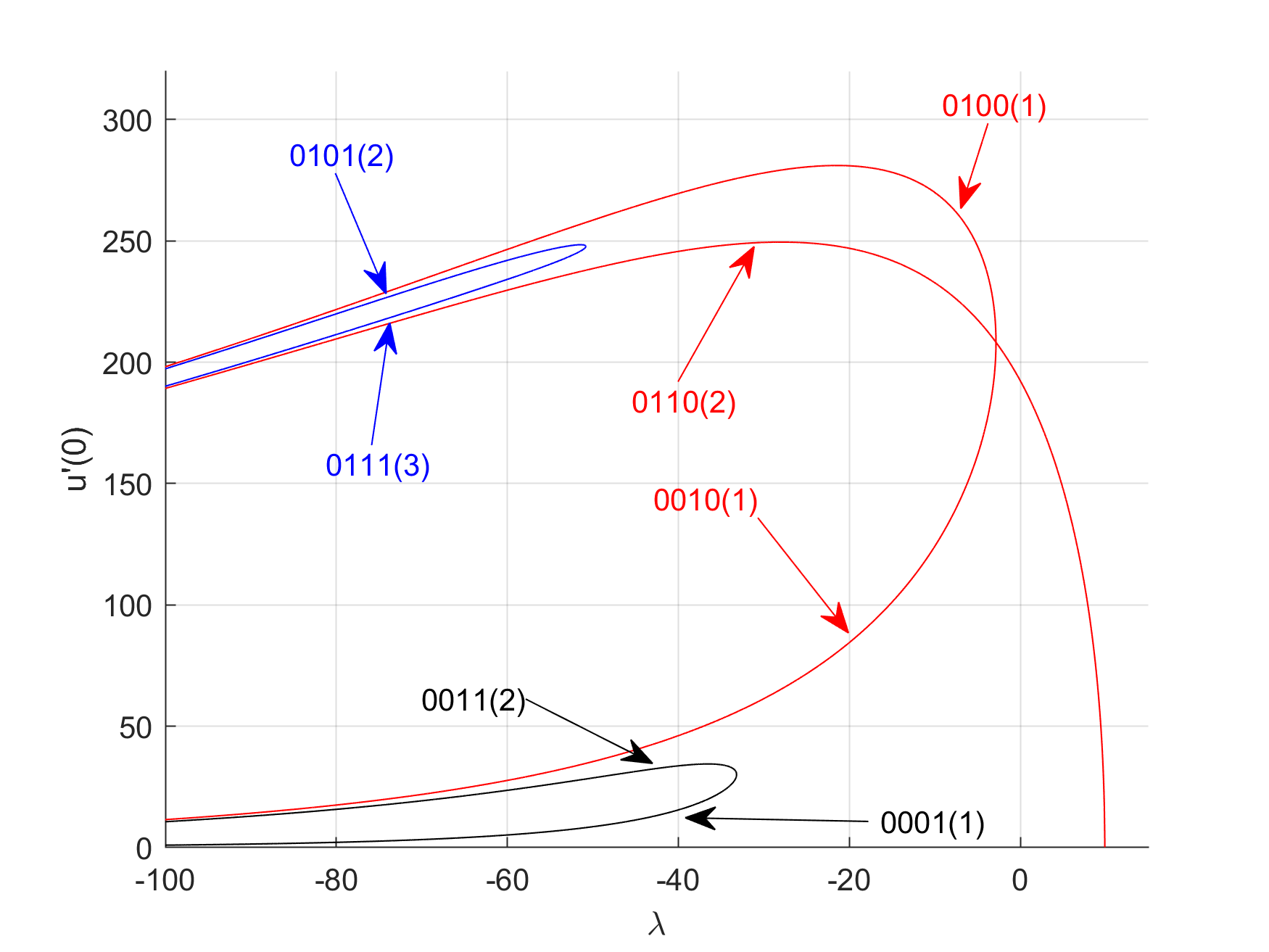}
\caption{Small positive solutions.}\label{FIG:sin7 bottom}
\end{subfigure}%
\begin{subfigure}[t]{0.47\textwidth}
\centering
\includegraphics[scale = 0.5]{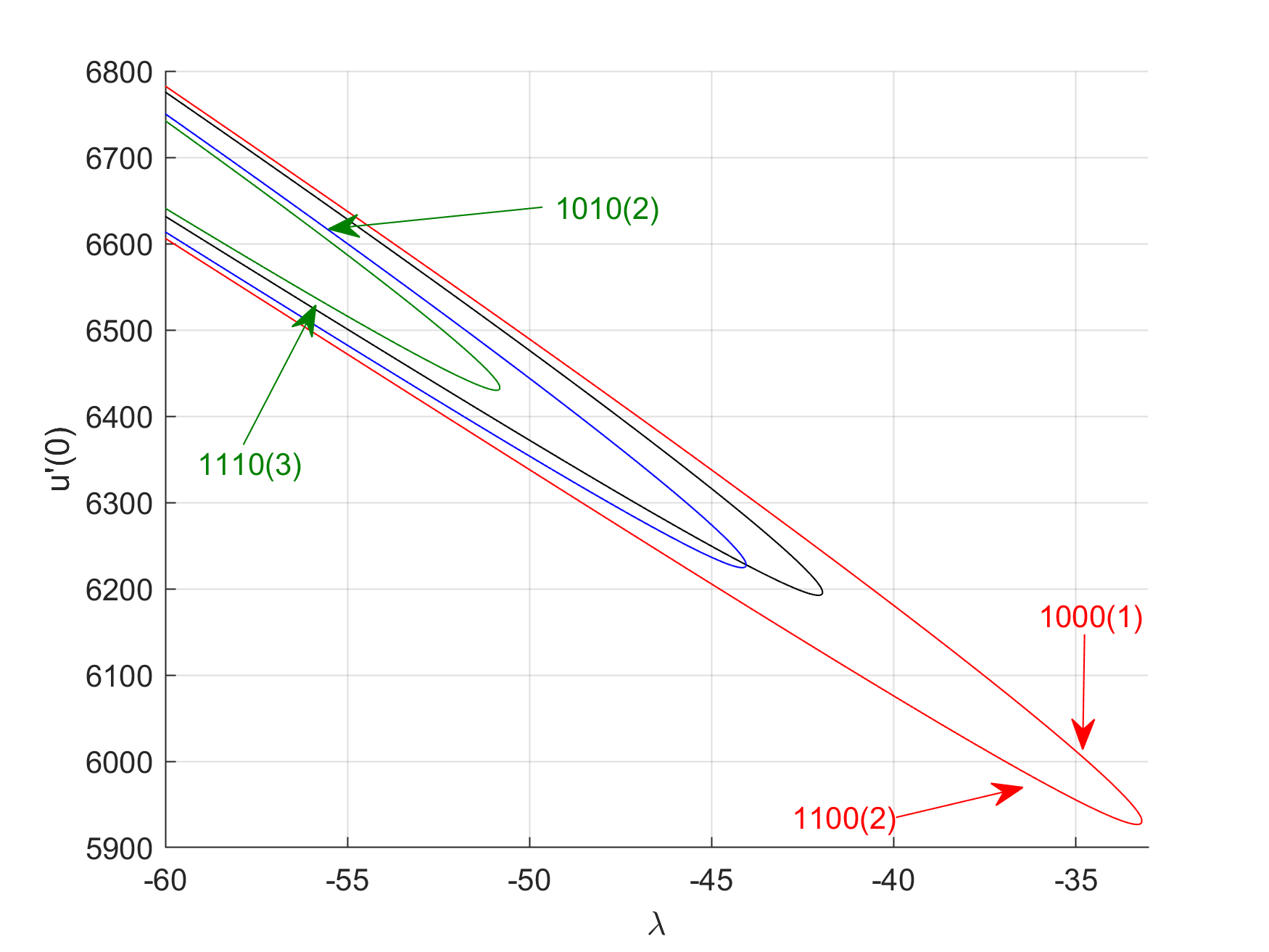}
\caption{Large positive solutions.}\label{FIG:sin7 top}
\end{subfigure}	
\begin{subfigure}[t]{0.47\textwidth}
\centering
\includegraphics[scale = 0.5]{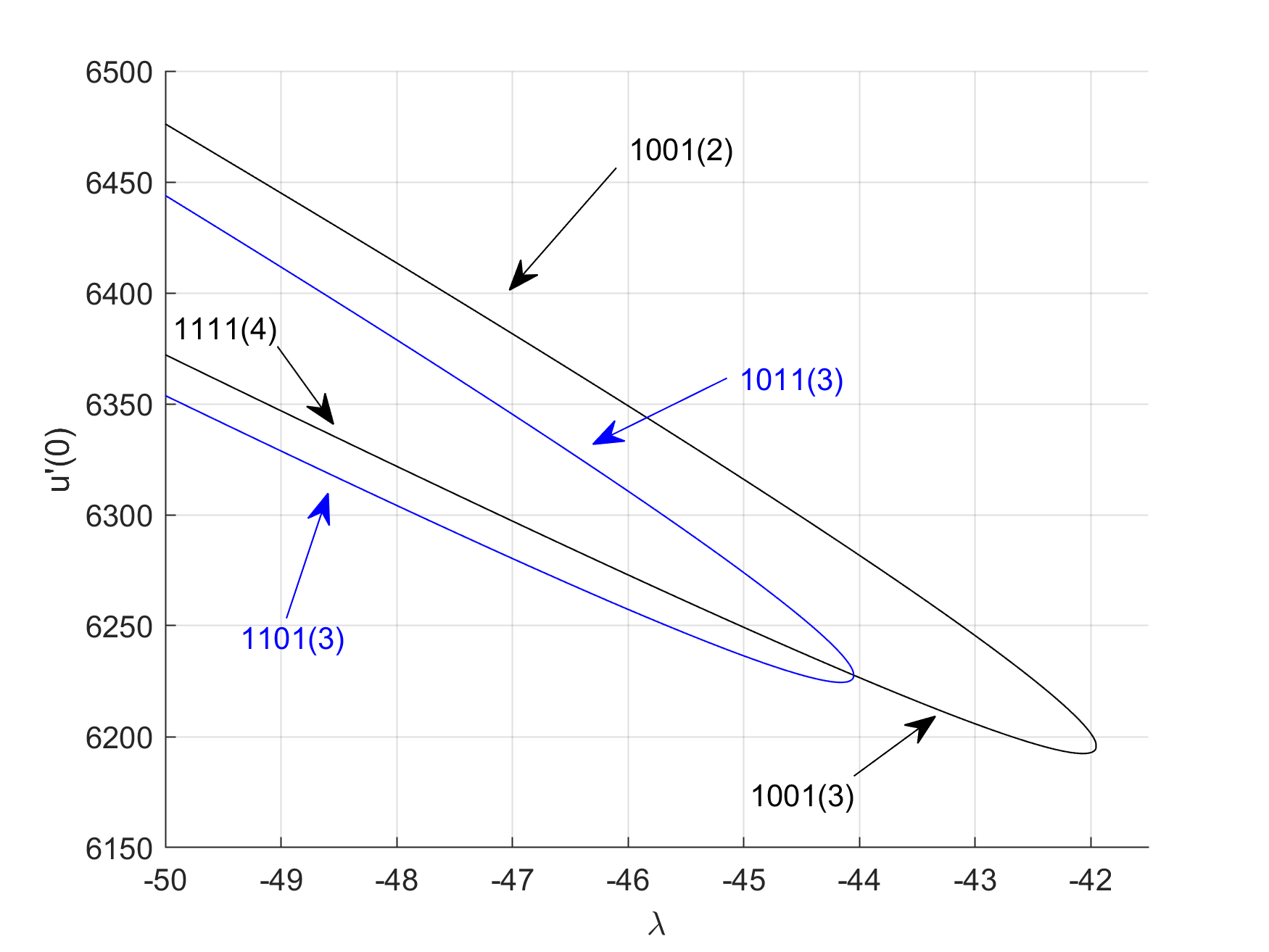}
\caption{A magnification of Figure \ref{FIG:sin7 top}.}\label{FIG:sin7 second bif}
\end{subfigure}	
\caption{Scattered bifurcation diagrams for  $a(x)=\sin(7\pi x)$.}
\label{FIG:sin7 global}
\end{figure}
Figure \ref{FIG:sin7 global}  consists of Figures \ref{FIG:sin7 bottom}, \ref{FIG:sin7 top} and \ref{FIG:sin7 second bif}, where we are plotting,
separately, the most significant branches of positive solutions that we have computed in our
numerical experiments. By simply looking at the ordinate axis in Figures \ref{FIG:sin7 bottom} and  \ref{FIG:sin7 top}, it is easily realized the ultimate reason why we are plotting these
components in two separate figures. Whereas for  those plotted on
the left $u'(0)<3\cdot 10^2 $, for those plotted on the right  we have that $u'(0)>59\cdot 10^2$.
So, plotting them in the same global bifurcation diagram would have pushed down against the
$\l$-axis all the branches on the left, much like in Figure \ref{FIG:sin5 complete diag}, but
straightening this pushing effect. Figure \ref{FIG:sin7 second bif} shows  a zoom of the secondary bifurcation arising  in Figure \ref{FIG:sin7 top}, to detail the types of the
positive solutions around it.
\par
Since $\mathscr{C}^+$ bifurcates from $u=0$ supercritically, by the exchange stability principle, \cite{CR2}, its solutions have Morse index zero until they reach the turning point. The bifurcation is very vertical in this case, hence, it is hard to determine for which $\l$ the turning point occurs. Anyway, Morse index increases to one as we pass the turning point and it remains the same until we reach the bifurcation point at $\lambda_s\approx -2.85$, where the Morse index becomes two for any smaller value of $\l$. By Theorem \ref{th2.1}, the solutions $(\l,u)\in \mathscr{C}^+$ with $\l\approx \pi^2$ have the form $s(\sin(\pi x)+y(s))$ for some $s>0$, $s\approx 0$. Thus, they have a single peak around $0.5$. Once crossed $\l_s$, these solutions
are of type $0110(2)$. The solutions along the bifurcated branches have types $0100(1)$ and
$0010(1)$, respectively. So, this piece of the global bifurcation diagram seems to
be generated by the two internal positive bumps of the weight function $a(x)$.
Besides the component $\mathscr{C}^+$,  Figure \ref{FIG:sin7 bottom} shows two additional global subcritical folds. The solutions on the lower half-branch of the inferior folding have type $0001(1)$
and change to type $0011(2)$ on its upper half-branch, as the turning point of this component is crossed. Similarly, the solutions on the lower half-branch of the superior folding have type $0111(3)$
and change to type $0101(2)$ on the upper one.  All those solutions can be generated, very easily,
by taking into account that its type must begin with a $0$, because $u'(0)$ is small,
while the remaining three digits should cover all the possible combinations of three elements taken from $\{0,1\}$. Thus, counting $u=0$, we have a total of $2^3=8$ solutions for sufficiently
negative $\l$.
\par
Analogously, Figure \ref{FIG:sin7 top} shows all solutions with $u'(0)$ sufficiently large, whose types
must begin with $1$. Thus, it also shows a total of $2^3=8$ solutions. According to our numerical
experiments, these solutions are distributed into three components. Namely, two isolated global subcritical folds, plus a third component consisting of two interlaced subcritical folds, which
is the component magnified and plotted in Figure \ref{FIG:sin7 second bif}. The bifurcation along this component occurs at $\lambda_s\approx -44.05$. 	
\par
According to these findings, based on a series of rather systematic numerical experiments,
the sum of the four digits of the type of the solutions, i.e., their number of peaks, always
provide us with the dimensions of their unstable manifolds, except for the solutions in a right neighborhood of the two bifurcation points on Figures \ref{FIG:sin7 bottom}   and
\ref{FIG:sin7 second bif}, where the solutions have types $0000(1)$ and $1001(3)$, respectively.
Nevertheless, for sufficiently negative $\l$, this is a general rule.\\
\begin{figure}[h!]		
	\centering
	\begin{subfigure}[t]{0.47\textwidth}
		\centering
		\includegraphics[scale = 0.5]{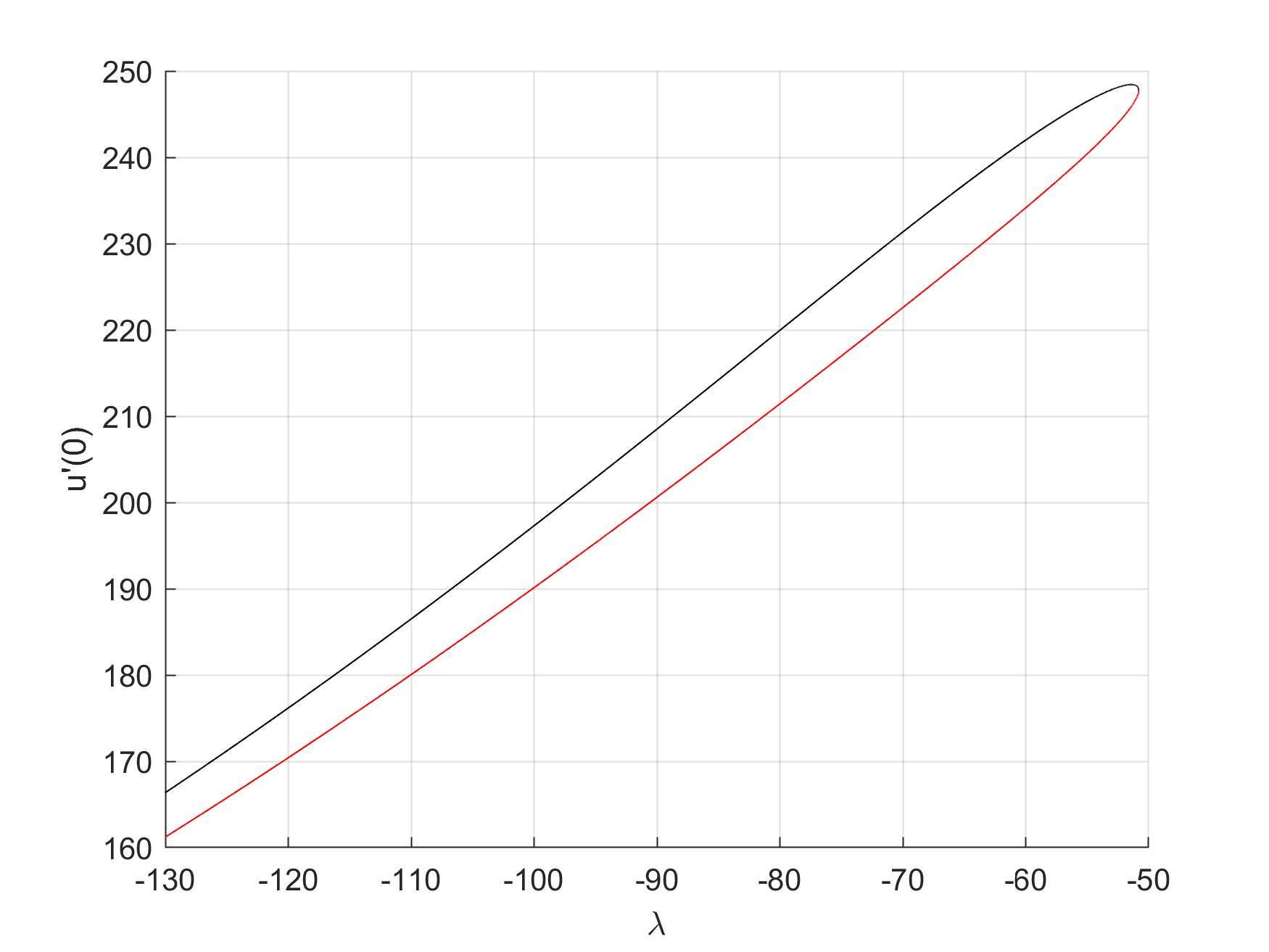}
		\caption{The blue upper fold of Figure \ref{FIG:sin7 bottom}.}\label{FIG:sin7_solutions_branch}
	\end{subfigure}%
	\begin{subfigure}[t]{0.47\textwidth}
		\centering
		\includegraphics[scale = 0.5]{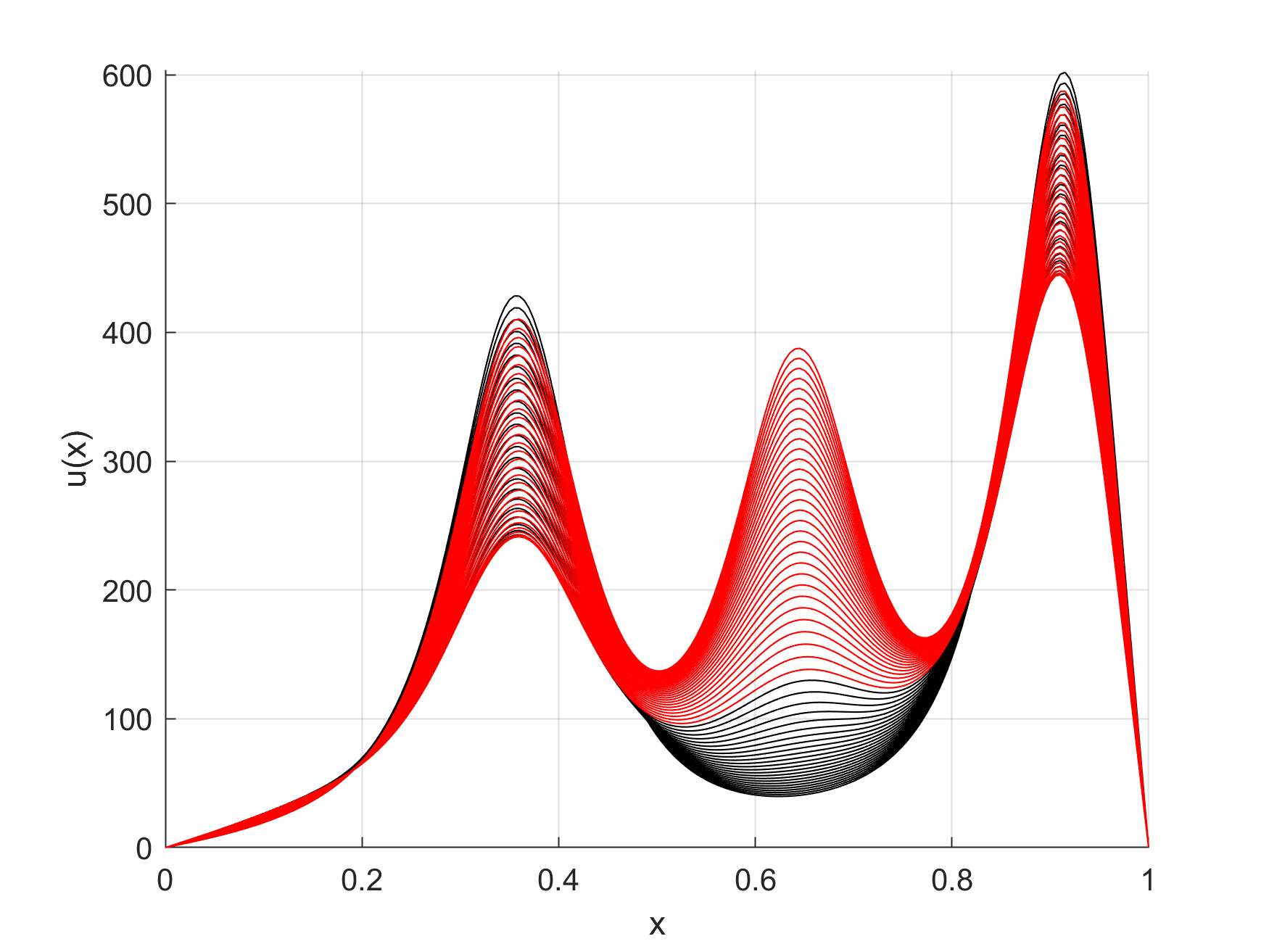}
		\caption{Plots of solutions of type $0111(3)$ (red) and $0101(2)$ (black) on the branch plotted on the left.}\label{FIG:sin7_solutions}
	\end{subfigure}
	\caption{Plots of solutions (right) along the branch (left) from Figure \ref{FIG:sin7 bottom}.}\label{FIG:sin7_solutions along branch}
\end{figure}
\indent In Figure \ref{FIG:sin7_solutions along branch} we have plotted a series of solutions of types $0111$ and $0101$ along the superior fold (blue branch) of Figure \ref{FIG:sin7 bottom}. The solutions on the lower half-branch are of type $0111$, because they exhibit three peaks,  and have been plotted in Figure \ref{FIG:sin7_solutions} using red color. As the turning point is approached,  the peaks of these solutions decrease until the central one is almost glued as the turning point is crossed. Once switched the turning point, the solutions need some additional, very short, room
for becoming of type $0101$ pure, since the central peak still persists for a while,
as it is illustrated in Figure \ref{FIG:sin7_solutions}, where those solutions have been plotted in black color. Essentially, as the turning point is switched, the external peaks of the $0111$ solutions increase, while the central peak is glued.

\section{The case $n=2$ with an additional parameter $\mu$}

\noindent In this section, we make the following choice
\begin{equation}
a(x):= \left\{ 	\begin{array}{ll} \mu \sin(5\pi x) & \quad \text{ if }\;\; x\in [0,0.2)\cup(0.8,1],
\\[1ex]  \sin(5\pi x) & \quad\text{ if }\;\; x\in [0.2,0.8],  \end{array}\right.
\end{equation}
where $\mu\geq 1$ is regarded as a secondary bifurcation parameter for \eqref{1.1}. The  behavior of this model for  $\mu = 1$ has been already described  in Section \ref{SEC:k=2}. The bifurcation direction is
\begin{equation*}
	D_{1} = -\frac{\sqrt{\frac{1}{2} \left(5-\sqrt{5}\right)} \left(5-\sqrt{5}\right)^2 (\mu-1)}{128 \pi}<0
\end{equation*}
for all $\mu>1$ and hence, the bifurcation is always subcritical.
\par
According to our numerical experiments, as we increase the value of $\mu$, the global bifurcation diagram remains very similar to the one plotted in Figures \ref{FIG:sin5 complete diag} and \ref{FIG:sin5 complete diag zoom},
up to reaching the critical value $\mu_{1} \approx 3.895$, where the global structure of the bifurcation
diagram changes. Figures \ref{FIG:sin5 mod mu 3.5} and \ref{FIG:sin5 mod mu 3.89} plot the
corresponding global bifurcation diagram for $\mu =3.5$ and $\mu=3.89$, respectively, whose global structure, topologically, coincides with the one already computed in Section \ref{SEC:k=2} for $\mu=1$.
\par
Essentially, as $\mu$ separates away from $\mu=1$ increasing towards $\mu=\mu_1$, the two subcritical folds lying in the upper part of the global bifurcation diagram plotted in Figure \ref{FIG:sin5 complete diag} are getting closer approaching the peak of the corresponding component $\mathscr{C}^+\equiv \mathscr{C}_\mu^+$, as well as the global subcritical folding beneath, as sketched in
Figure \ref{Figmuscale}.
\par
According to our numerical experiments, at the critical value of the parameter $\mu_1$, the set of positive solutions of \eqref{1.1} consists of two components, instead of four, because three of the previous four components of the problem for $\mu<\mu_1$ are now touching at a single point
playing the role of a sort of \emph{organizing center} with respect to the secondary parameter $\mu$, whereas the upper interior supercritical folding remains separated away from $\mathscr{C}_{\mu_1}^+$. Naturally, $\mathscr{C}_{\mu_1}^+$ consists of $\lim_{\mu\ua \mu_1} \mathscr{C}_\mu^+$
plus the limits of the  previous exterior upper folds and folds beneath $\mathscr{C}_\mu^+$
for $\mu<\mu_1$.
\begin{figure}[h!]		
		\centering
		\begin{subfigure}[t]{0.47\textwidth}
			\centering
			\includegraphics[scale = 0.5]{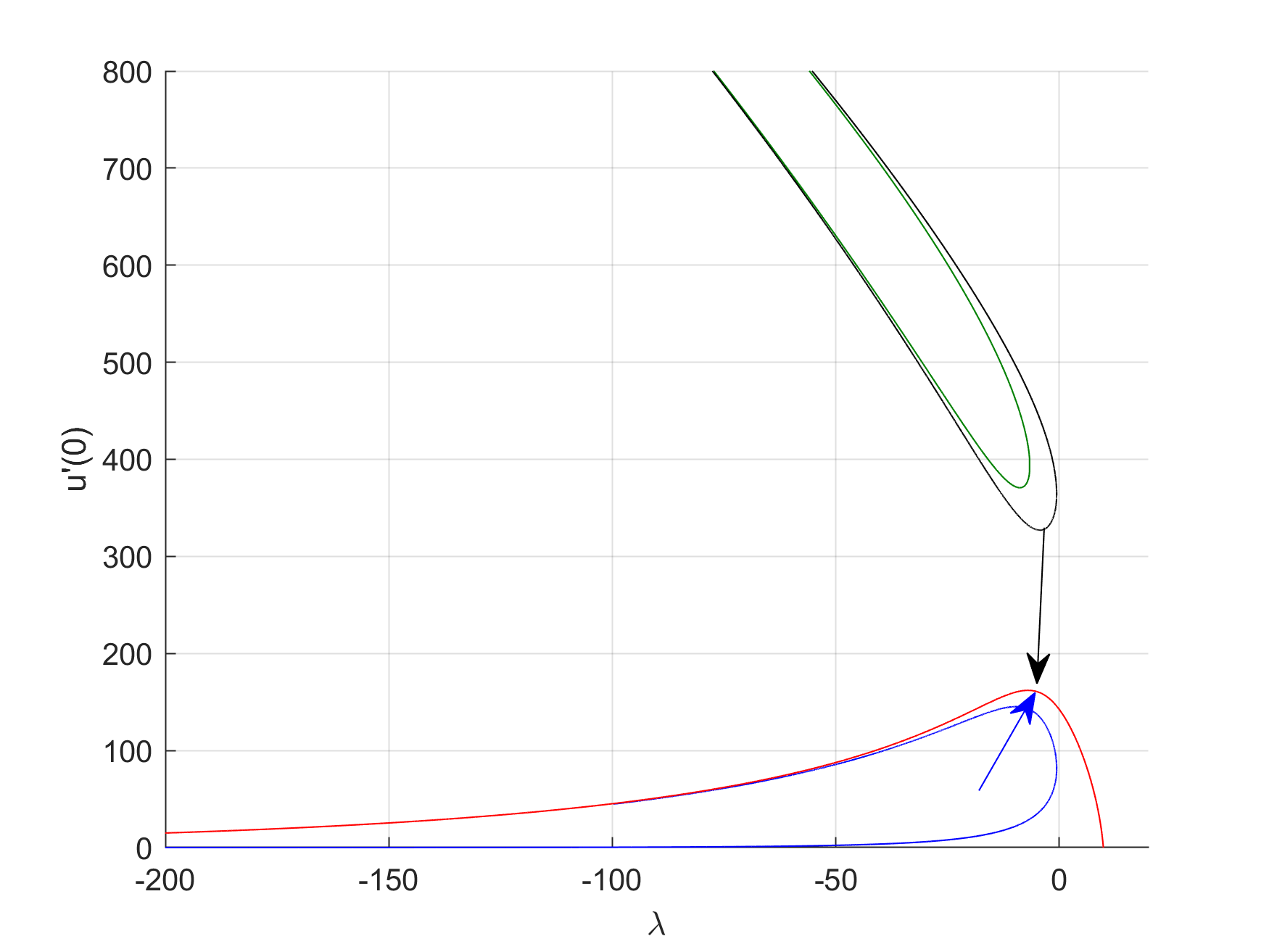}
			\caption{$\mu = 3.5<\mu_{1}$}\label{FIG:sin5 mod mu 3.5}
		\end{subfigure}%
		\begin{subfigure}[t]{0.47\textwidth}
			\centering
			\includegraphics[scale = 0.5]{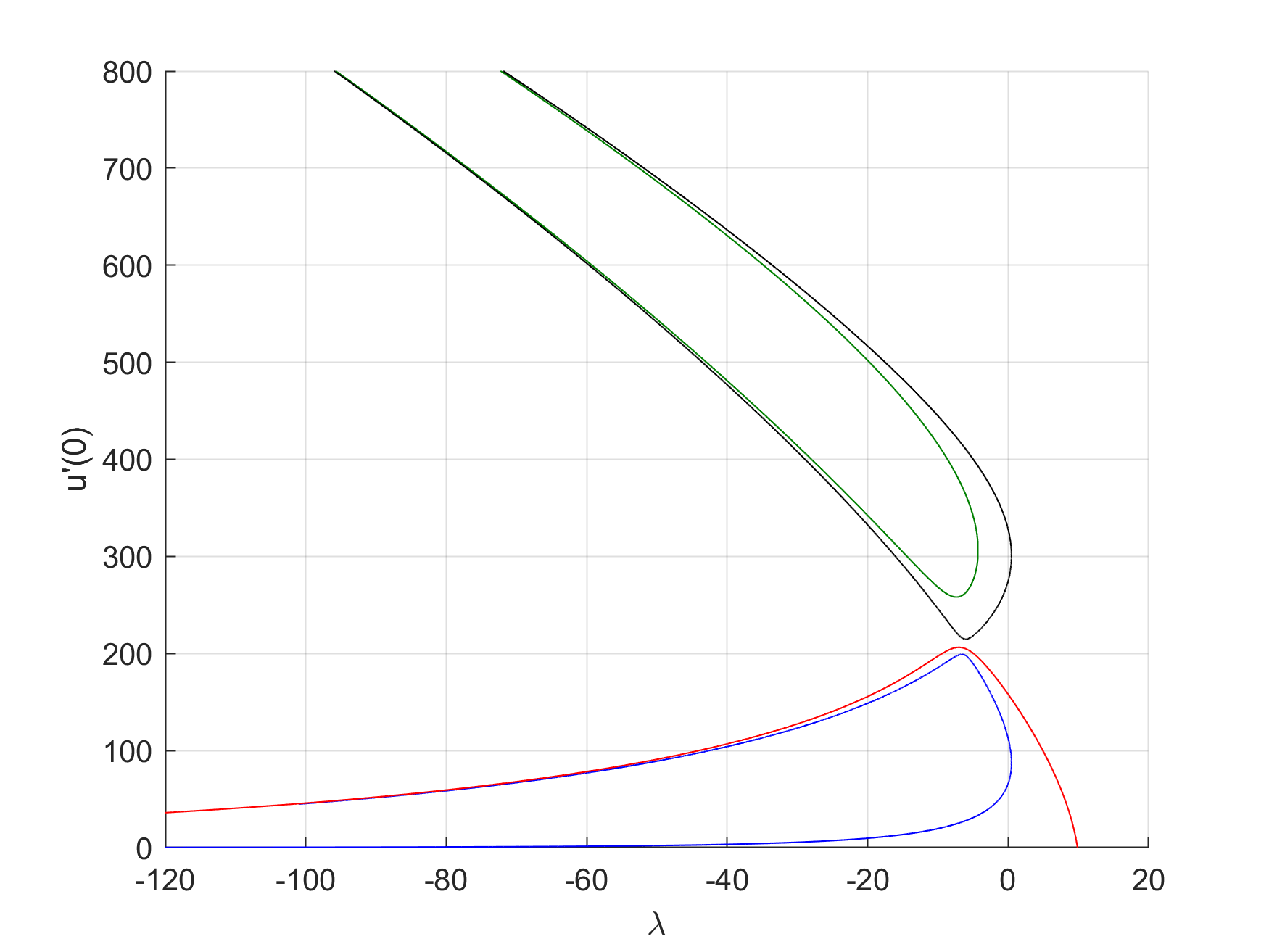}
			\caption{$\mu = 3.89<\mu_{1}$}\label{FIG:sin5 mod mu 3.89}
		\end{subfigure}
		\caption{Global bifurcation diagrams for $\mu<\mu_{1}$.}
\label{Figmuscale}
\end{figure}
\par
The numerics suggests that, as $\mu$ increases separating away from $\mu_1$, the touching point
of the old three components spreads out into two secondary bifurcation points from the new component $\mathscr{C}_\mu^+$, in such a way that the \lq\lq previous\rq\rq \, folds do now bifurcate from
$\mathscr{C}_\mu^+$ at these two bifurcation values with respect to the primary parameter $\l$, say $\l_1(\mu)>\l_2(\mu)$, as illustrated in Figure \ref{FIG:sin5 mod mu 3.92}, where it becomes apparent how the old upper interior folding component still remains separated away from $\mathscr{C}_\mu^+$.
Essentially, the upper half-branch of the old folding above $\mathscr{C}_\mu^+$ together with the lower
half-branch of the old interior folding provide us with the branch bifurcating from $\mathscr{C}_\mu^+$
at $\l_2(\mu)$, for  $\mu>\mu_1$, whereas the lower half-branch of the old folding above
$\mathscr{C}_\mu^+$ together with the upper half-branch of the old interior folding provide us with
the new branch bifurcating from $\mathscr{C}_\mu^+$ at $\l_1(\mu)$. And this situation persists for all $\mu\in (\mu_1,\mu_2)$, where  $\mu_{2}\approx 3.925$. The bigger is $\mu$ in the interval $(\mu_1,\mu_2)$, the more separated stay the two bifurcation values $\l_1(\mu)$ and $\l_2(\mu)$ and the more approaches the exterior upper fold to the component $\mathscr{C}_\mu^+$. The separation between
the bifurcation values is very well  illustrated by the next table that provides us with the corresponding values of $\l_1(\mu)$ and $\l_2(\mu)$ for three values of $\mu$ in $(\mu_1,\mu_2)$:

\begin{table}[h!]
	\begin{tabular}{|c|c|c|c|}
		\hline
		$\mu$ & 3.9 & 3.91 & 3.92 \\
		\hline
		$\lambda_{1}(\mu)$ & -5.1186 & -4.4513 & -3.9938 \\
		\hline
		$\lambda_{2}(\mu)$ & -7.5845 & -8.4129 & -9.0284 \\
		\hline
	\end{tabular}
\vspace{0.4cm}
	\caption{$\lambda_i(\mu)$ for three values of $\mu\in(\mu_{1},\mu_{2})$.}\label{TABLE:secondary bifurcation values}
\end{table}

\begin{figure}[h!]		
\centering
\begin{subfigure}[t]{0.47\textwidth}
\centering
\includegraphics[scale = 0.5]{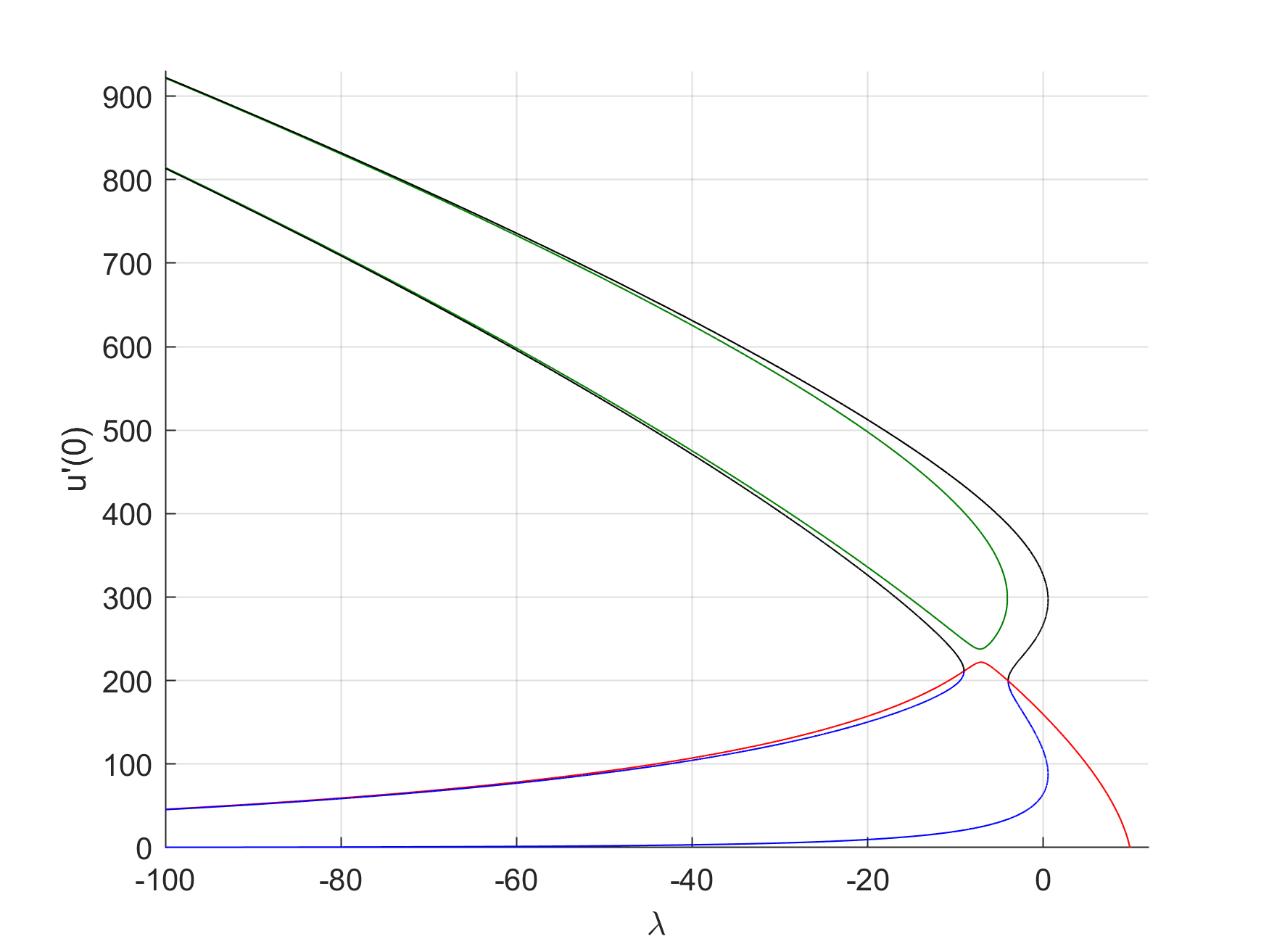}
\caption{$\mu_{1}<\mu = 3.92<\mu_{2}$}\label{FIG:sin5 mod mu 3.92}
\end{subfigure}%
\begin{subfigure}[t]{0.47\textwidth}
\centering
\includegraphics[scale = 0.5]{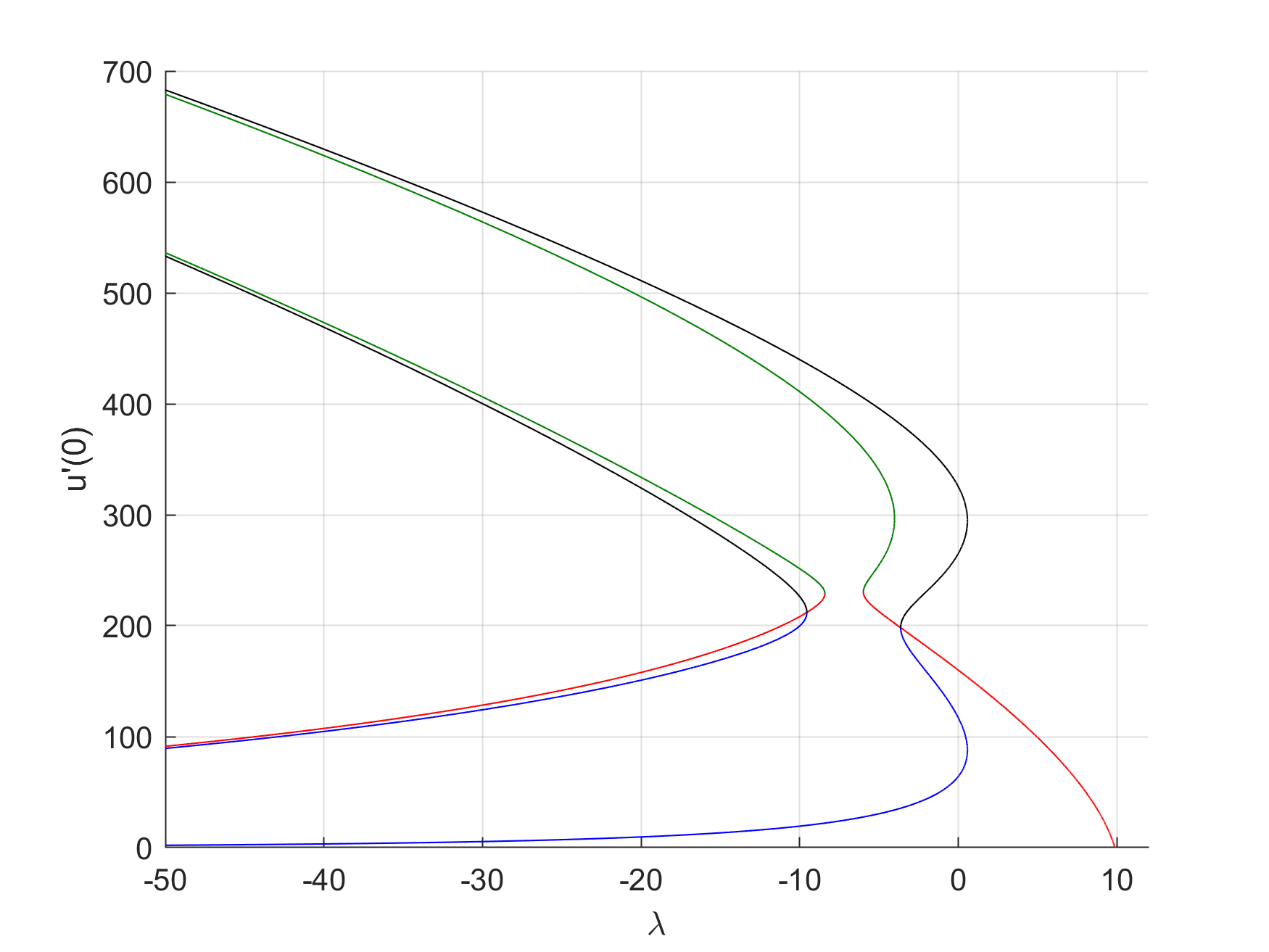}
\caption{$\mu_{2}<\mu = 3.93$}\label{FIG:sin5 mod mu 3.93}
\end{subfigure}
\caption{Two significant global bifurcation diagrams}
\end{figure}

And this situation persists,  until $\mu$ reaches the critical value $\mu_2$, where, according to the numerics,  the exterior folding touches $\mathscr{C}_{\mu_2}^+$ at a single point, in such a way
that the set of positive solutions of \eqref{1.1} consists of the single component
$\mathscr{C}_{\mu_2}^+$. As $\mu>\mu_2$ separates away from $\mu_2$, our numerical experiments
provide us with the global bifurcation diagram plotted in Figure \ref{FIG:sin5 mod mu 3.93},
where, once again, the set of positive solutions of \eqref{1.1} consists of two components,
$\mathscr{C}_\mu^+$ plus a global subcritical fold with a bifurcated secondary branch with the
structure of a global subcritical folding. Thus, a new re-organization of the previous solution branches has occurred through a sort of mutual re-combination.
\par
The global bifurcation diagrams remained topologically equivalent for all values of
$\mu >\mu_2$ for which we computed them. Figures \ref{FIG:sin5 mod mu 3.93} and \ref{FIG:sin5 mod mu 4.5} plot them for $\mu=3.93$ and $\mu=4.5$, respectively. In both cases, the set of positive solutions consists of $\mathscr{C}_\mu^+$ plus two global subcritical folds that meet at a single point, which can be viewed as a secondary bifurcation point from any of them.  This structure persists
for any further larger values of $\mu$. Figure \ref{FIG:sin5 mod mu 4.5 magnified} shows a magnification of the most significant parts  of Figure \ref{FIG:sin5 mod mu 4.5} superimposing the individual
types of the solutions together with the dimensions of their unstable manifolds.
\begin{figure}[h!]		
\centering
\begin{subfigure}[t]{0.47\textwidth}
\centering
\includegraphics[scale = 0.5]{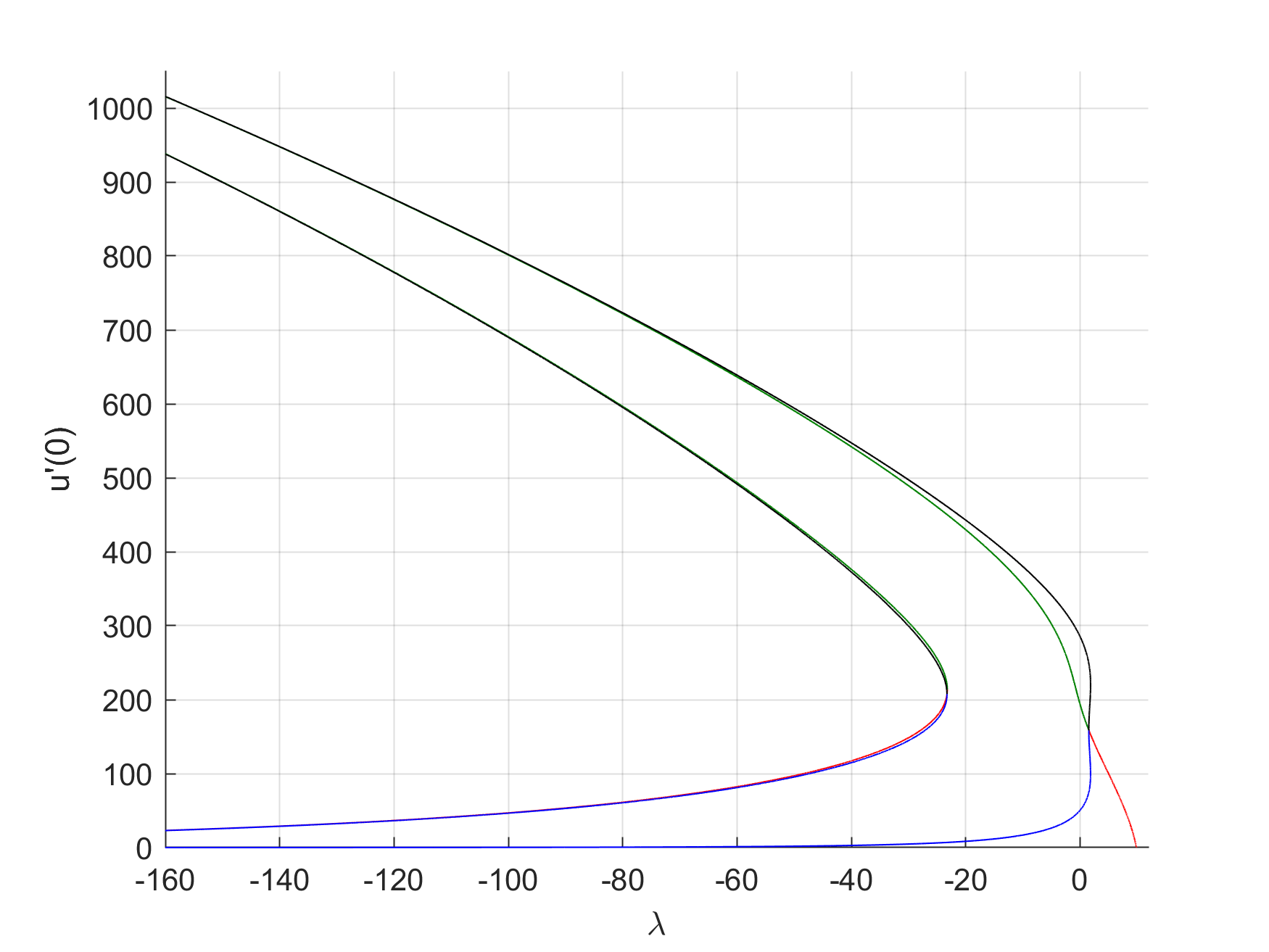}
\caption{$\mu_{2}<\mu = 4.5$}\label{FIG:sin5 mod mu 4.5}
\end{subfigure}%
\begin{subfigure}[t]{0.47\textwidth}
\centering
\includegraphics[scale = 0.5]{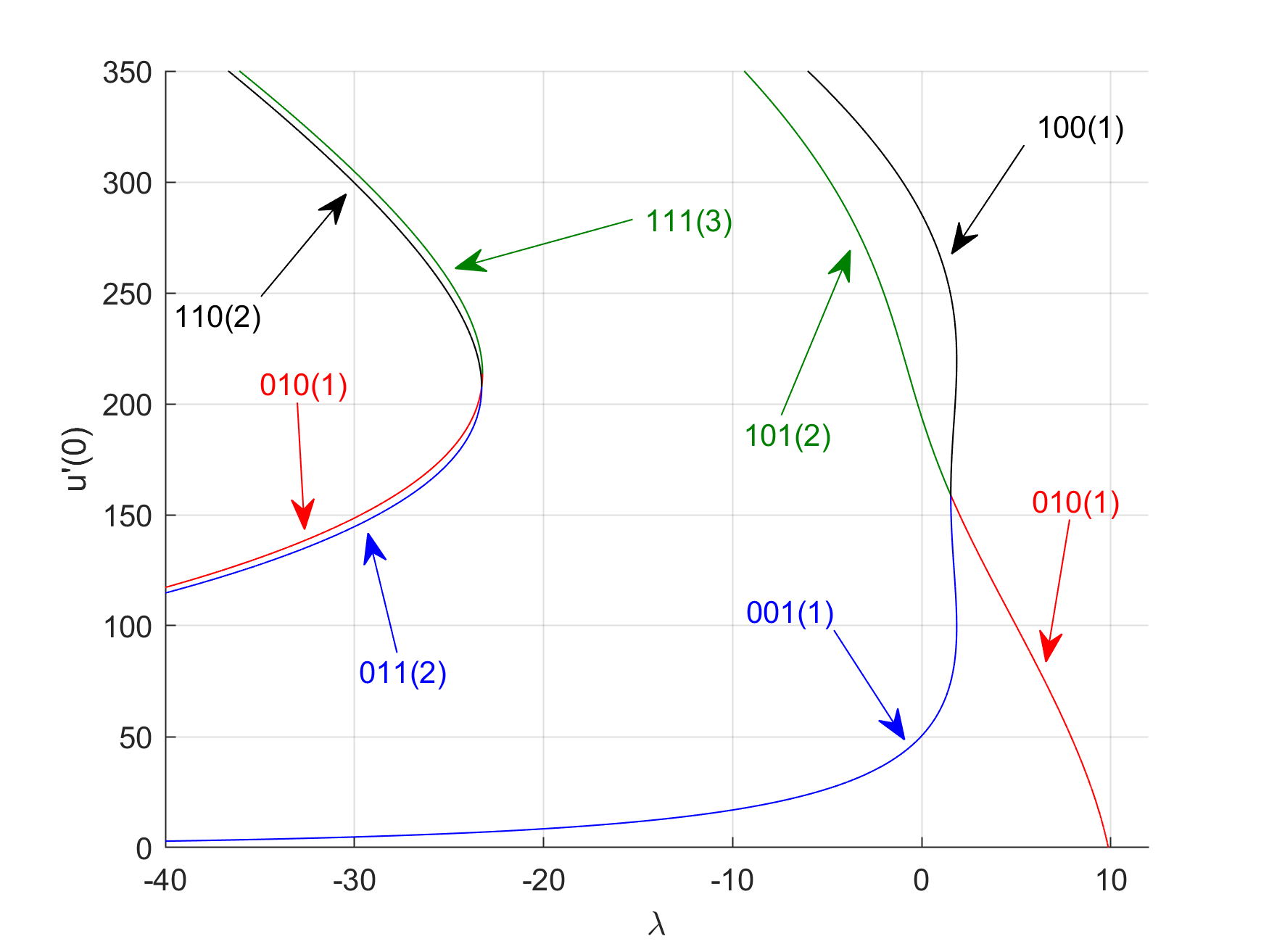}
\caption{A zoom of the plot on the left}\label{FIG:sin5 mod mu 4.5 magnified}
\end{subfigure}
\caption{Global bifurcation diagram for $\mu>\mu_{2}$}
\end{figure}

In full agreement with Conjecture \ref{con3.1}, for every $\mu\geq 1$, there exists $\l(\mu)<0$ such that
\eqref{1.1} has $2^3-1=7$ positive solutions for every $\l \leq \l(\mu)$. For the choice $\mu=4.5$,
$\l(\mu)\approx -23.27$ equals the $\l$-coordinate of the bifurcation point of the global subcritical folds. Note that, for this special choice, \eqref{1.1} possesses three solutions at $\l=0$.\\
\begin{figure}[h!]		
	\centering
	\begin{subfigure}[t]{0.47\textwidth}
		\centering
		\includegraphics[scale = 0.5]{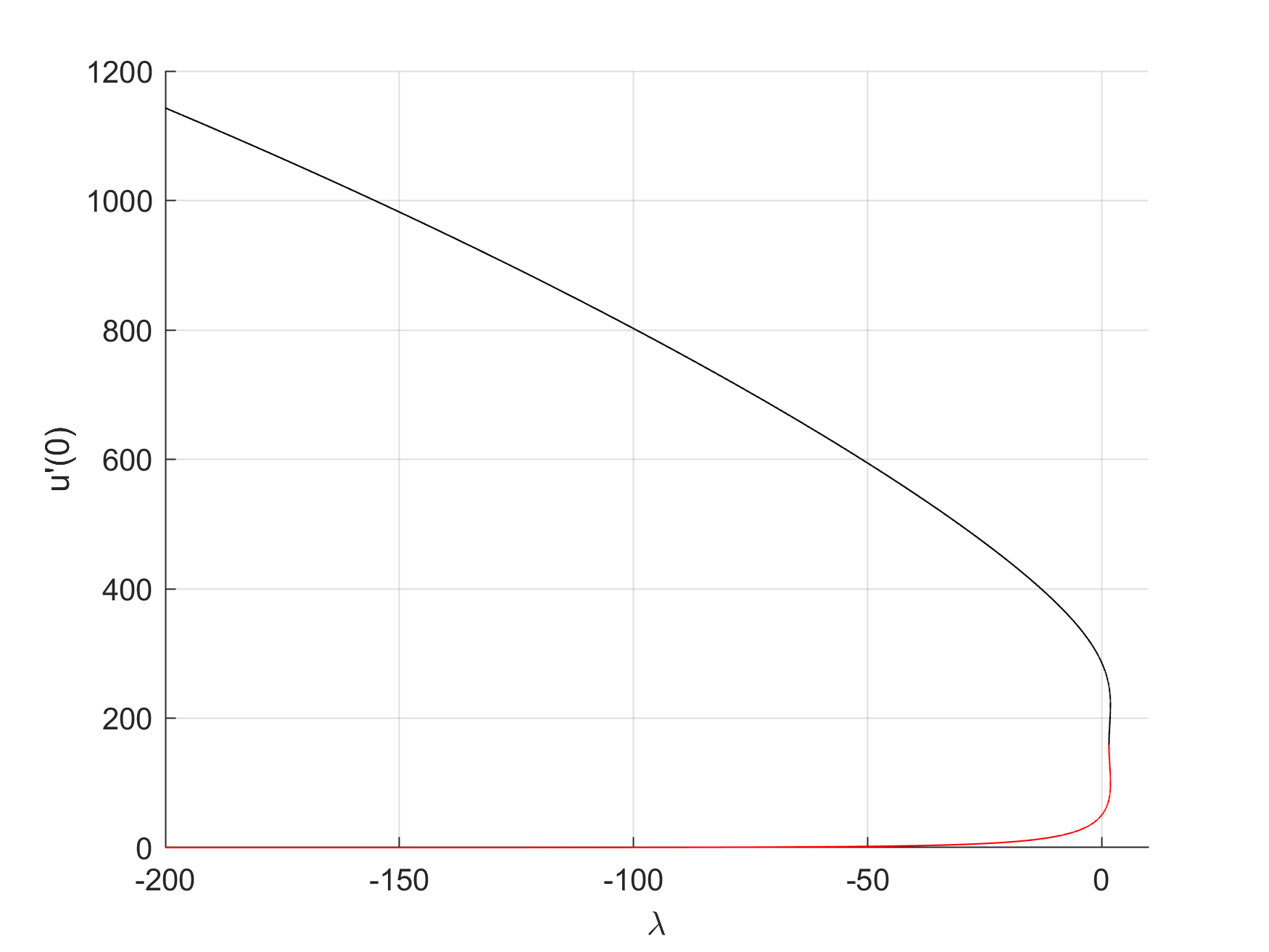}
		\caption{The branch bifurcating from $\mathscr{C}_{\mu}^+$.}\label{FIG:sin5_mu_solutions_branch}
	\end{subfigure}%
	\begin{subfigure}[t]{0.47\textwidth}
		\centering
		\includegraphics[scale = 0.5]{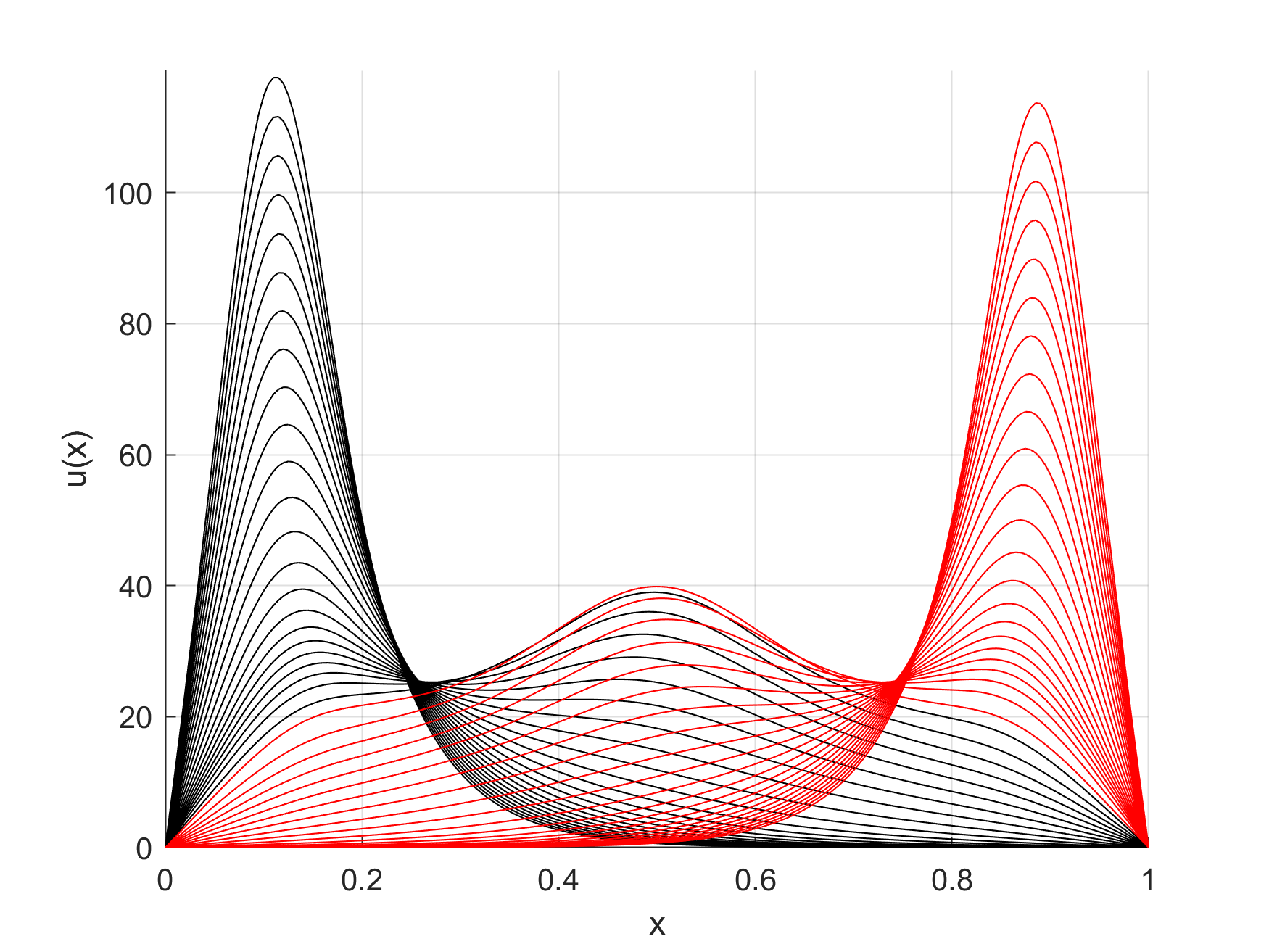}
		\caption{Plots of solutions $001(1)$ and $100(1)$ on the left.}\label{FIG:sin5_mu_solutions}
	\end{subfigure}
	\caption{Plots of solutions (right) along the branch (left) from Figure \ref{FIG:sin5 mod mu 4.5}.}\label{FIG:sin5_mu_solution_along_branch}
\end{figure}
\indent Figure \ref{FIG:sin5_mu_solution_along_branch} plots a series of solutions with types $100$ and $001$ along the blue/black branch of Figure \ref{FIG:sin5 mod mu 4.5 magnified} that is part of the component $\mathscr{C}_\mu^+$, which has been isolated
in Figure \ref{FIG:sin5_mu_solutions_branch}. According to our numerical experiments, the solutions on
the lower half-branch are of type $001$; they have been plotted using red color. As the branching point is approached (see Figure \ref{FIG:sin5 mod mu 4.5 magnified}), the right peak diminishes and the solution
looks like the first eigenfunction $\sin(\pi x)$. At the branching point, the solution changes its old type to $100$. These solutions have been plotted using black color. The peak on the left starts to increase as we separate away from the bifurcation value.

\section{Numerics of bifurcation problems}

\noindent To discretize \eqref{1.1} we have used two methods. To compute the small positive solutions bifurcating
from $u=0$ we implemented a pseudo-spectral method combining a trigonometric
spectral method with collocation at equidistant points, as in G\'{o}mez-Re\~{n}asco and L\'{o}pez-G\'{o}mez  \cite{GRLGJDE,GRLGNA},  L\'{o}pez-G\'{o}mez, Eilbeck, Duncan and Molina-Meyer \cite{LGEDMM}, L\'{o}pez-G\'{o}mez and Molina-Meyer \cite{LGMMJDE,LGMMTPB,LGMMMCS},   L\'{o}pez-G\'{o}mez, Molina-Meyer and Tellini \cite{LGMMT}, L\'opez-G\'omez, Molina-Meyer and Rabinowitz \cite{LGMMR}, and Fencl and L\'{o}pez-G\'{o}mez \cite{FLG}. This gives high accuracy
at a rather reasonable computational cost (see, e.g., Canuto, Hussaini, Quarteroni and Zang \cite{CHQZ}). However, to  compute the large positive solutions we have preferred  a centered finite differences scheme, which gives high accuracy at a lower computational cost, as it  is runs  much faster in computing global solution branches in the bifurcation diagrams.
\par
The pseudo-spectral method is more efficient and versatile for choosing the shooting direction from the trivial solution in order to compute the small positive solution of $\mathscr{C}^+$, as well as to detect the bifurcation points along the solution branches. Its main advantage in accomplishing this task
comes from the fact that it provides us with the true bifurcation values from the trivial solution, while the scheme in differences only gives a rough approximation to these  values.
A pioneering reference on these methods is the paper of Eilbeck \cite{Ei}, which was seminal
for the research teams  of the second author.
\par
For computing all the global subcritical folds  arisen along this paper, we adopted the following, rather novel, methodology. Once computed $\mathscr{C}^+$, including all bifurcating branches from
the primary curve emanating from $u=0$ at $\l=\pi^2$, one can ascertain  the types of the solutions
in $\mathscr{C}^+$ for sufficiently negative $\l$. As, due to Conjecture \ref{con3.1} and the argument
supporting it in Section 1, we already know that
\eqref{1.1} admits $2^{n+1}-1$ positive solutions for sufficiently negative $\l$, together with their
respective types, we can determine the types of all solutions of \eqref{1.1} for $\l$ sufficiently negative that remained outside the component $\mathscr{C}^+$. Suppose, e.g., that we wish to compute
the solution curve containing the positive solutions of type $011(2)$ in Figure \ref{FIG:sin5 mod mu 4.5 magnified}. Then, we consider as the initial iterate, $u_0$,  for the underlying Newton method some function with a similar shape. If the choice is sufficiently accurate, after finitely many iterates, whose number depends on how far away stays  from the true solution the initialization $u_0$, the Newton scheme should provide us with the first positive solution on that particular component. Once located the first point, our numerical path-following codes provide us with the entire solution curve almost algorithmically though the code developed by Keller and Yang \cite{KY} to treat the turning points of these folds as if they were regular points treated with the implicit function theorem.
\par
The huge complexity of some of the computed bifurcation diagrams, as well as their deepest
quantitative features, required an extremely careful control of all the steps in the involved subroutines. This explains why the available commercial  bifurcation solver packages, such as AUTO-07P,  are
almost un-useful to deal with differential equations, like the one of \eqref{1.1},
with heterogeneous coefficients. As noted by Doedel and Oldeman in \cite[p.18]{DO},
\vspace{0.2cm}
\par
\begin{small}\emph{ \lq \lq given the non-adaptive spatial discretization,
the computational procedure here is not appropriate for PDEs with solutions that rapidly
vary in space, and care must be taken to recognize spurious solutions and bifurcations.\rq\rq}
\end{small}
\par
\vspace{0.2cm}
This is nothing than one of the main problems that we found in our numerical experiments, as the number of critical points of the solutions increases according to the dimensions
of their unstable manifolds, and the solutions grew up to infinity as $\l\da -\infty$
within the intervals of $\mathrm{supp\,}a^+$, while, due to Theorem \ref{th3.1}, they
decayed to zero on the intervals where $a(x)$ is negative. Naturally, for all numerical methods is a serious challenge to compute solutions exhibiting simultaneously internal and boundary layers,
where the gradients can oscillate as much as wish for sufficiently negative $\l$.
\par
For general Galerkin approximations, the local convergence of the solution paths at regular, turning and simple bifurcation points was proven by Brezzi, Rappaz and Raviart in \cite{BRR1,BRR2,BRR3} and by L\'{o}pez-G\'{o}mez,  Molina-Meyer and  Villareal \cite{LGMMV} and L\'{o}pez-G\'{o}mez,  Eilbeck, Duncan and  Molina-Meyer in \cite{LGEDMM} for  codimension two singularities in the context  of systems. In these situations, the local structure of the solution sets for the continuous and
 the discrete models are known to be equivalent.
 \par
 The global continuation solvers used to compute the solution curves of this paper, as well as the
dimensions of the unstable manifolds of all the solutions filling them, have been built from the
theory on continuation methods of  Allgower and Georg \cite{AG},
Crouzeix and Rappaz \cite{CrRa}, Eilbeck \cite{Ei}, Keller \cite{Ke}, L\'{o}pez-G\'{o}mez \cite{LG88} and L\'{o}pez-G\'{o}mez, Eilbeck, Duncan and Molina-Meyer \cite{LGEDMM}.

\section{Final discussion}

\noindent Our systematic numerical experiments have confirmed that the following features should be true
for a general $a(x)$ with $\mathrm{supp\,}a^+$ consisting of $n+1$ intervals separated away by $n$
intervals where $a$ is negative:
\begin{itemize}
\item As $\l\da -\infty$, \eqref{1.1} has, at least, $2^{n+1}-1$ positive solutions.  Theorem \ref{th3.3} has shown it under condition \eqref{iii.33}. It remains an open problem ascertaining whether, or not, \eqref{iii.33} holds. Actually, for the special choice \eqref{1.3}, it should have exactly $2^{n+1}-1$.
\item  As $\l\da -\infty$, the Morse index of any positive solution $u$ of type
 $d_1d_2\cdots d_{n+1}$, $d_j\in\{0,1\}$, is given by
$$
  \mathscr{M}(u):= \sum_{j=1}^{n+1}d_j.
$$
\item The eventual symmetric character of the solutions cannot be lost along any of the components of the set of solutions, unless a bifurcation point is crossed. Thus, each fold consists of either symmetric solutions around $0.5$, or asymmetric ones.
\end{itemize}
Note that in the special case when $a(x)$ is given by \eqref{1.3}, $a(0.5)<0$ if $n$ is odd, while $a(0.5)>0$ if $n$ is even. Moreover, according to our numerical experiments, the component $\mathscr{C}^+$ does not admit any bifurcation point if $n=2$, whereas it admits one if $n\in\{1,3\}$. Thus, one might be tempted to believe that, in general, $\mathscr{C}^+$ should not have any bifurcation  point if $a(0.5)<0$. Our numerical experiments in Section 7 show that, for the special choice
\begin{equation}
\label{9.1}
	a(x):= \left\{ 	\begin{array}{ll} 4.5 \sin(5\pi x) & \quad \text{ if }\;\; x\in [0,0.2)\cup(0.8,1],
\\[1ex]  \sin(5\pi x) & \quad\text{ if }\;\; x\in [0.2,0.8],  \end{array}\right.
\end{equation}
the component $\mathscr{C}^+$ has a bifurcation point (see Figure \ref{FIG:sin5 mod mu 4.5 magnified}), though $a(0.5)<0$. Thus, one should be extremely careful in conjecturing anything from either numerical experiments, or heuristical considerations, as they might drive, very easily, to extract false conjectures (see Sovrano \cite{So}).
\par
According to the numeric of Section 7, the problem
\begin{equation}
\label{9.2}
		\left\{\begin{aligned}
			-u'' &= a(x)u^{2}\quad\text{in }(0,1),\\
			u(0) &= u(1) = 0,
		\end{aligned}\right.
\end{equation}	
for the special choice \eqref{9.1} has, exactly, three positive solutions, and it should
not admit anymore, by consistency with the structure of the global
bifurcation diagram  (see Figure \ref{FIG:sin5 mod mu 4.5 magnified}). Thus, Corollary 1.4.2
of Feltrin \cite{Fe} is optimal in the sense that it cannot be satisfied for sufficiently small
$\mu>0$. Precisely, it does not hold when $\mu=1$ for the special choice \eqref{9.2}. Since \eqref{1.1} still possesses $2^3-1=7$ positive solutions for sufficiently negative $\l<0$, this example also shows the independence between the conjecture of \cite{GRLGJDE} and the multiplicity results of Feltrin and  Zanolin \cite{FZ} and Feltrin \cite{Fe}.
\par
For every $\mu \in (3.895,3.925)$, the component $\mathscr{C}^+$ of \eqref{1.1} for the choice
\begin{equation}
\label{9.3}
	a(x):= \left\{ 	\begin{array}{ll} \mu \sin(5\pi x) & \quad \text{ if }\;\; x\in [0,0.2)\cup(0.8,1],
\\[1ex]  \sin(5\pi x) & \quad\text{ if }\;\; x\in [0.2,0.8],  \end{array}\right.
\end{equation}
exhibits two bifurcation points along it. This might be the first example of this nature
documented in the abundant  literature on superlinear indefinite problems.
\par
As, generically, higher order bifurcations break down by the eventual asymmetries
of the weight functions, as discussed in Chapter 7 of \cite{LG01} and in \cite{LGT}, we conjecture that
\begin{itemize}
\item Generically, when $a(x)$ is asymmetric about $0.5$, the set of positive solutions
of \eqref{1.1} consists of the component $\mathscr{C}^+$ plus $n$ supercritical folds, $\mathscr{D}_j$, $j\in \{1,...,n\}$, in such a way that, as $\l\da -\infty$, \eqref{1.1} admits, at least, one solution in $\mathscr{C}^+$ and two solutions in
$\mathscr{D}_j$ for each $j\in\{1,...,n\}$.
\end{itemize}
The global bifurcation diagram plotted in Figure \ref{FIG:sin5 complete diag zoom} being a paradigm
of this global topological behavior.
\par
The analysis of the reorganization in components of the positive solutions of \eqref{1.1} carried out in Section 7 for the special choice \eqref{9.3} when $\mu$ increases from $3.89$ up to reach the value $\mu=3.93$ reveals the high complexity that the  global bifurcation diagrams of \eqref{1.1} might have when $a(x)$ changes of sign a large number of times by incorporating the appropriate control parameters into the problem. Getting any insight into
this problem is a challenge.

\end{document}